\documentclass[12pt]{amsart}

\usepackage[margin=1in]{geometry}
\usepackage[utf8]{inputenc}
\usepackage[english]{babel}
\usepackage{amssymb}
\usepackage{amsmath}
\usepackage{amsthm}
\usepackage{dsfont}
\usepackage{amsmath}
\usepackage{cite}
\usepackage{xcolor}

\usepackage [english]{babel}
\usepackage [autostyle, english = american]{csquotes}
\MakeOuterQuote{"}

\usepackage{hyperref}
\hypersetup{
    colorlinks,
    linkcolor={blue!80!black},
    citecolor={blue!60!black},
    urlcolor={blue!80!black},
}
\newtheorem{thm}{Theorem}[section]

\newtheorem{lem}[thm]{Lemma}
\newtheorem{prop}[thm]{Proposition}

\theoremstyle{definition}

\theoremstyle{definition}
\newtheorem{defi}[thm]{Definition}
\newtheorem*{definition-non}{Definition}

\theoremstyle{remark}
\newtheorem*{rmkx}{Remark}
\newenvironment{rmk}
  {\pushQED{\qed}\rmkx}
  {\popQED\endrmkx}

\newcommand*\interior[1]{\mathring{#1}} 

\newcommand{\spt}{\mathord{\mathrm{spt}}}

\newcommand{\eps}{\varepsilon}
\newcommand{\emps}{\varnothing}

\newcommand{\rn}{\mathbb R^n}

\newcommand{\sn}{S^{n-1}}

\newcommand{\dn}{D^{n-1}}

\newcommand{\kn}{\mathcal K^n}

\newcommand{\kno}{\mathcal K^n_o}

\newcommand{\bla}{\raise.2ex\hbox{$\scriptstyle\pmb \langle$}}
\newcommand{\sbla}{\raise.1ex\hbox{$\scriptscriptstyle\pmb \langle$}}
\newcommand{\bra}{\raise.2ex\hbox{$\scriptstyle\pmb \rangle$}}
\newcommand{\sbra}{\raise.1ex\hbox{$\scriptscriptstyle\pmb \rangle$}}

\newcommand{\balpha}{\pmb{\alpha}}

\numberwithin{equation}{section}

\begin{document}

\title{The Uniqueness of the Gauss Image Measure}

\author{Vadim Semenov}
\address{
Courant Institute of Mathematical Sciences, New York University, 251 Mercer St., New York, NY 10012}
\curraddr{}
\email{vs1292@nyu.edu}
%\thanks{
%The author was supported by ...}

%\subjclass[2010]{Primary }

\date{\today}

\dedicatory{}


\maketitle

\begin{abstract}

We show that if the Gauss Image Measure of submeasure $\lambda$ via convex body $K$ agrees with the Gauss Image Measure of $\lambda$ via convex body $L$, then the radial Gauss Image maps of their duals, $\balpha_{K^*}$ and $\balpha_{L^*}$, are equal to each other almost everywhere as multivalued maps with respect to $\lambda$. As an application of this result, we establish that, in this case, dual bodies, $K^*$ and $L^*$, are equal up to a dilation on each rectifiable path connected component of the support of $\lambda$. Additionally, we provide many previously unknown properties of the radial Gauss Image map, most notably its variational Lipschitz behavior, establish some measure theory concepts for multivalued maps and, as a supplement, show how the main uniqueness statement neatly follows from the Hopf Theorem under additional smooth assumptions on $K$ and $L$.

\end{abstract}

\vspace{30mm}
\tableofcontents
\newpage

\section{Introduction}\label{Section Intro}
%TODO Add citations, rewrite INTRO
%TODO grammar check??
%TODO adress rectifiable  path connected.
%TODO why looking at duals is better

The history of the Gauss Image Problem originates with the paper of Aleksandrov who first wondered whether there exists a convex body (or a polytope) with prescribed Aleksandrov's integral curvature (or prescribed values of exterior angles). Aleksandrov solved this problem completely by providing necessary and sufficient existence conditions on the Aleksandrov's integral curvature (or values of exterior angles) as well as showing that the solution is unique up to scaling. Later, K. J. B\"or\"oczky, E. Lutwak, D. Yang, G. Y. Zhang and Y. M. Zhao in \cite{GIP}  introduced a measure-theoretic generalization of the Aleksandrov problem involving \textit{the Gauss Image Measure}. Formally speaking, given any submeasure $\lambda$ and a convex body $K$, the Gauss Image Measure $\lambda(K,\cdot)$ is a pullback of submeasure $\lambda$ under the radial Gauss Image map $\balpha_K$, which is a composition of the usual Gauss map of the convex body $K$ with a radial map of $K$. There are, in fact, many measures associated with convex bodies which turn out to be the Gauss Image Measure for different submeasurers $\lambda$. The integral curvature defined by Aleksandrov \cite{Aleks}, surface are measures of Aleksandrov-Fenchel-Jessen \cite{Aleks0} and, more recently, the dual curvature measures \cite{HLYZ16}, are all the Gauss Image measures. In the attempt to classify these type of measure the following problem was introduced in \cite{GIP}: 

 \smallskip
 
 \noindent{\textbf{The Gauss Image problem}}  \textit{Suppose $\lambda$ is a submeasure defined on the Lebesgue measurable subsets of $\sn$ and $\mu$ is a Borel measure on $\sn$. What are necessary and sufficient conditions on $\lambda$ and $\mu$, so that there exists a convex body $K$ with origin in its interior such that \begin{equation}
  \mu=\lambda(K,\cdot)?
\end{equation}
If such a convex body exists, to what extent is it unique?}  

 \smallskip
 
Many people have established relevant results for this problem. Let us start with mentioning a series of papers by Aleksandrov who solved the problem for uniform Lebesgue measure $\lambda$, see \cite{Aleks1} and \cite{Aleks}. Different proofs were given by Oliker \cite{Oliker} and Bertrand \cite{Bertrand}. The $L_p$ analogues of the Aleksandrov problem (The $L_p$ Gauss Image Problem with $\lambda$ uniform Lebesgue measure) were considered by Huang, Lutwak, Yang and Zhang in \cite{Lp Aleks}, by Mui in \cite{Stephanie}, and by Zhao in \cite{Zhao}. The $L^p$ analog of the Gauss Image Problem was considered in \cite{Lp Gauss} by C. Wu, D. Wu, and N. Xiang and in \cite{Smooth solutions to the Gauss image problem} by L. Chen, D. Wu and N. Xiang.

While there are many results on generalizations of the Gauss Image Problem (or its more restricted case, the Aleksandrov Problem) to the $L_p$ setting, the original version of the Gauss Image Problem is still not fully resolved. The absolutely continuous case of this problem was addressed in \cite{GIP}. There, K.J. B\"or\"oczky, E. Lutwak, D. Yang, G. Y. Zhang and Y. M. Zhao introduced the simple geometric condition on measures  $\mu$ and $\lambda$ called \textit{the Aleksandrov relation of $\mu$ and $\lambda$} that was shown to be a sufficient condition for the existence of the body $K$ under absolutely continuous assumption for $\lambda$. Moreover, they established that if $\lambda$ is assumed to be absolutely continuous and strictly positive on open sets, the Aleksandrov relation is a necessary assumption and the solution is unique up to a dilation. 

On the opposite side, the fully discrete version of the Gauss Image Problem was addressed by V. Semenov in \cite{Semenov GIP}, who introduced \textit{the weak Aleksandrov relation}, a necessary assumption for the Gauss Image Problem. There, it was found that the existence of the body corresponds to the uniqueness of the function maximizing \textit{the Assignment functional} over the finite set of, what is called, \textit{the Assignment functions}, which parametrize all possible assignments of weights $\lambda$ to weights $\mu$ for discrete measures $\mu$ and $\lambda$. Then, the existence of solution was obtained for $\mu$ and $\lambda$ satisfying the geometric \textit{Edge-normal loop free condition}. It turned out, that for discrete measure the solution body is always non-unique if exists; yet, fixing $\mu$ and $\lambda$, if $K$ and $L$ are solutions, then their corresponding assignment functions are the same. This turned out to be a proper restatement of the uniqueness question for the discrete Gauss Image Problem.

In the attempt to relax results of \cite{GIP} to the weak Aleksandrov relation, the existence of the solution to the Gauss Image Problem under weak Aleksandrov relation, $\mu$ discrete and $\lambda$ absolutely continuous was obtained by V. Semenov in \cite{Semenov WAC}. The question of whether there exists a body $K$, solving $\mu=\lambda(K,\cdot)$ with $\lambda$ absolutely continuous and weak Aleksandrov relation between $\mu$ and $\lambda$ is still open and can be considered the most important unsolved question with regards to the Gauss Image Problem.

Before we proceed to the main statements of this work, let us also briefly mention works related to the regularity of the Gauss Image Problem. The regularity of the original Aleksandrov Problem was investigated by P. Guan and Y. Li in \cite{GL97com} and V. Oliker in \cite{Ol2} and \cite{Ol21}. The recent result of Q. Li, W. Sheng and X. Wang can be found to be also related to the Gauss Image Problem\cite{flow}. Finally, for a more general viewpoint on regularity, one should consult works of L. Caffarelli \cite{Caffarelli}; S.-Y. Cheng and S.-T. Yau  \cite{Cheng}; and N. S. Trudinger and X.-J. Wang \cite{Trudinger}.

\vspace{6mm}

 In this work, we would like to address the uniqueness question of the Gauss Image Problem in its most general setting. As shown in \cite{Semenov GIP} when measures $\mu$ and $\lambda$ are discrete the solution body is not unique. Nevertheless, it turns out that one can obtain a uniqueness statement for their inverse radial Gauss Image Maps. The following is the main result of the paper. Here, $\kno$ denotes the set of convex sets in $\mathbb R^n$ that contain 0 in the interior. $K^*$ is the dual body of $K$. By $\balpha_{K^*}$ and $\balpha_{L^*}$ we denote the radial Gauss Image maps of bodies $K^*$ and $L^*$. Note that submeasure is a measure where we relax countable additivity to countable subadditivity, see Definition \ref{submeasure defenition}.
 %TODO write down inequality. TODO What should I do with this? It is well-known that the uniqueness of the Minkowski problem, the conceptual dual of the Gauss Image Problem, is related to the equality condition of the Brunn-Minkowski inequality. Similar inequality exist with relation to the Gauss Image Problem so that the equality part of it is equivalent to the uniqueness of the Gauss Image Problem. Original work of Aleksandrov also proved the uniqueness separately, only later to use for showing existence of the body with prescribed Aleksandrov curvature (using invariance of the domain argument) so the general uniqueness statement can be seen as a step towards the full resolution of the existence question for the Gauss Image Problem. 
 
 \begin{thm}\label{main}
 Let $K,L\in \kno$. Suppose $\lambda(K,\cdot)$ and $\lambda(L,\cdot)$ are finite Borel measures for a spherical submeasure $\lambda$, defined on the Lebesgue measurable subsets of $\sn$. Then, $\lambda(K,\cdot)=\lambda(L,\cdot)$ if and only if $\balpha_{K^{*}}$ and $\balpha_{L^{*}}$  are equal almost everywhere as multivalued maps with respect to submeasure $\lambda$. 
 \end{thm}

 %TODO, so does it look better this way?
 
 %TODO I should think about two things. First is how, whether this result can help solving existence statement. Secondly, rewriting uniqueness in terms of functional inequality. Will it involve volume?
Since $\balpha_{K^{*}}$ and $\balpha_{L^{*}}$ are multivalued valued maps, one has to be careful about the definition of multivalued maps being equal almost everywhere. We address this issue in Section \ref{Section Measure}. The proper reformulation of Theorem \ref{main} is that $\lambda(K,\cdot)=\lambda(L,\cdot)$ if and only if $\forall \omega\subset S^{n-1}$ Borel sets $\lambda(\balpha_K(\omega)\triangle\balpha_L(\omega))=0$, where $\triangle$ denotes the symmetric difference between two sets. Another way to state Theorem \ref{main} is the following:
\begin{equation}\label{another interpretation}
\begin{gathered}
	\forall \omega \text{ Borel sets } \lambda(\balpha_K(\omega))=\lambda(\balpha_L(\omega)). \\ \Leftrightarrow \\ \forall \omega \text{ Borel sets } \balpha_K(\omega)=\balpha_L(\omega) \text{ up to a $\lambda$ measure zero set.}
\end{gathered}
\end{equation}
One can easily see, that this is quite special behavior of the radial Gauss Image maps. For example, equation \eqref{another interpretation} doesn't hold if one considers different rotations of the sphere instead of $\balpha_K$ and $\balpha_L$, for uniform measure $\lambda$. 

Before we state our next results, with a more familiar geometric interpretation, let us address the importance of the proof behind Theorem \ref{main}. Our approach to Theorem \ref{main} was to first analyze the uniqueness question under the additional smooth assumptions on $K$ and $L$. With a non-trivial argument, we managed to reduce the study of the uniqueness of the Gauss Image Problem to the investigation of images of Borel sets under the radial Gauss Image maps of the harmonic mean variation for bodies $K$ and $L$, see Definition \ref{Defenition of Harmonic Mean Variation} and Lemma \ref{same for strictly}. It turned out, that if $K$ and $L$ are $C^1$ strictly convex bodies, this study essentially reduces to a proper application of the Hopf Theorem, which establishes Theorem \ref{main}. 

However, when dropping $C^1$ strictly convex assumption, the radial Gauss Image Map suddenly becomes multi-valued, which greatly complicates the application of the differential topological tools to the problem. Yet, we managed to mimic the application of the Hopf Theorem by establishing Lipschitz behavior of the mentioned, harmonic mean variation, for the Hausdorff distance between sets on the sphere, see Proposition \ref{continuity of sum}. The established techniques and ideas seem very important, as they provide a way to apply differential topological results to the setting with very little regularity. They are especially important if the objects are not necessarily continuous in a regular sense and maps are multivalued which, for instance, often happens when working with normals of surfaces of less than $C^1$ regularity.

Let us also mention that Theorem \ref{main} combined with a continuity of the radial Gauss Image Map, the Proposition \ref{continuity}, implies the following continuity result for simultaneous normals of bodies $K$ and $L$. Here, $u_\delta$ and 
$\balpha_{K^*,L^*}(u)_\eps$ stands for the outer parallel sets of respectively $u$ and $\balpha_{K^*,L^*}(u)$, that is their fattening by $\delta$ and $\eps$. See \eqref{outer parallel set} for the formal definitions. The $\spt \lambda$ denotes the support of submeasure $\lambda$, see Definition \ref{Definition of supporrt}.

% implies that for $u$ in the support of $\lambda$ the set $
% 	\balpha_{K^*,L^*}(u):=\balpha_{K^*}(u)\cap\balpha_{L^*}(u)$ is never empty and $\balpha_{K^*,L^*}(u)$ changes continuously on the support of $\lambda$.   
 	
 	\begin{thm}\label{new additional main}
 		Let $K,L\in \kno$. Suppose $\lambda(K,\cdot)=\lambda(L,\cdot)$ are finite Borel measures for a spherical Lebesgue submeasure $\lambda$. Then, given $u\in\spt \lambda$,	
 		\begin{equation}
 			\balpha_{K^*,L^*}(u):=\balpha_{K^*}(u)\cap\balpha_{L^*}(u)\neq \emps
 		\end{equation}
 		In particular, $\balpha_{K^*,L^*}$ defined on $\spt \lambda$ is a continuous map. That is, for any $\eps>0$ there exist $\delta >0$ such that for any $u\in\spt \lambda$ \begin{equation}
 		\balpha_{K^*,L^*}(u_\delta)\subset \balpha_{K^*,L^*}(u)_\eps.
 	\end{equation}
 	\end{thm}

In particular, one can view this as a very similar statement to the equality of $\balpha_{K^*}$ and $\balpha_{L^*}$. To see this, consider two convex functions $f$ and $g$ on the real line. Let $\partial f$ and $\partial g$ be their subdifferential maps (subderivative) and suppose for any $x\in\mathbb R$ we have $\partial f(x)\cap \partial g(x)\neq \emps$. Then, $f$ is equal to $g$ up to a constant. This is essentially as saying that the integral of the right derivative of a convex function has to agree with the integral of the left derivative of a convex function. While there is no reference to this result one can look at Corollary 24.2.1 in \cite{Rock}. In particular, this implies that $\partial f(x)= \partial g(x)$ for all $x\in\mathbb R$. To see the similarity, one should view the inverse Gauss Image maps of bodies $K$ and $L$ as the gradient of the support functions of convex bodies $K$ and $L$. 

Note that, for any convex bodies $K$ and $L$, $\balpha_K$ and $\balpha_L$ are singular valued outside of the Lebesgue measure zero set. Thus, if $\spt \lambda=\sn$, then outside of the Lebesgue measure zero set $\balpha_{K^*}(u)\cap\balpha_{L^*}(u)\neq \emps$ implies $\balpha_{K^*}(u)=\balpha_{L^*}(u)$, which immediately establishes that $K$ is equal to $L$ up to scaling. This is exactly the uniqueness result established in \cite{GIP} and in the works of Aleksandrov \cite{Aleks1}, \cite{Aleks0}, \cite{Aleks}.

 In less formal language, Theorems \ref{main} and Theorem \ref{main 2}  claim that if $\lambda(K,\cdot)=\lambda(L,\cdot)$ then bodies are "equal" (up to scaling) on parts where sets of normals have positive measure and we can make any deformations to bodies at the boundaries where normals have measure zero. The formal statement for this intuition is in the next Theorem. Here, by $\spt\lambda$ we denote the usual support of submeasure $\lambda$, see Definition \ref{Definition of supporrt}. For the formal definition of rectifiable path connected set see Definition \ref{rectifiable path connected}.

\begin{thm}\label{main 2}
	Let $K,L\in\kno$. Suppose $\lambda(K,\cdot)=\lambda(L,\cdot)$ are finite Borel measures for a spherical submeasure $\lambda$, defined on the Lebesgue measurable subsets of $\sn$. Then on each rectifiable path connected component $D\subset \spt \lambda$, $K^*$ and $L^*$ are are equal up to a dilation. Alternatively, for each $v_1,v_2\in D$ we have \begin{equation}
		\frac{h_K(v_1)}{h_L(v_1)}=\frac{h_K(v_2)}{h_L(v_2)},
		\end{equation}
		where by $h_K$ and $h_L$ we denote the support functions of $K$ and $L$.
\end{thm}
\begin{rmk}
	In particular, one can think about this in terms of tangential bodies. We also note that any connected open set on $\sn$ is automatically rectifiable path connected.
\end{rmk}

As already mentioned, if $\lambda$ is positive on open sets, Theorem \ref{main 2}  implies the uniqueness (up to dilation) of the solution for the Gauss Image Problem obtained in \cite{GIP} as well as the uniqueness (up to dilation) for the Aleksandrov's Problem. This also justifies the intuitive claim about separately scaling parts in \cite{Semenov WAC}, see section \ref{Section Intro}, and implies the uniqueness of the assignment functional, see Proposition 5.3 in \cite{Semenov GIP}.

The natural question for Theorem \ref{main 2} is whether the same result can hold just for any connected set $D\subset \spt \lambda$ without the need for rectifiable path connectedness. The problem here lies in the possible pathological behavior of support of $\lambda$, which might have a fractal nature. As an illustration of similar behavior, consider intricate Whitney's example of a $C^1$ function $f:\mathbb R^2\rightarrow \mathbb R$ with $\nabla f=0$ on a path connected set $D\subset \mathbb {R}^2$ such that $f$ is not constant on $D$. See \cite{Whitney} for details. We address this issue and other complications in the beginning of Section \ref{Section Appl}. It would be very interesting to see if something like this can happen for the radial Gauss Image map. That is, can one construct a measure $\lambda$ supported on a connected set $D$, such that $K^*$ is not a dilate of $L^*$ on $D$.

% To conclude, let us briefly comment on the proofs of the following results. The main difficulty in establishing Theorem \ref{main} is complicated structure of the radial Gauss Image map as it can only be thought of as a multi-valued map with weak notion of continuity. Interestingly enough, if you assume that $K$ and $L$ are $C^1$ strictly convex bodies in Theorem \ref{main}, thus making radial Gauss Image maps singular valued, the result follows by an application of Hopf Theorem. However, without this assumption we have to battle the problem using different tools which in the end mimic the Hopf Theorem for a more general class of objects. The established techniques seem important, especially if one wants to use standard topological arguments, but the objects are not necessarily continuous in a regular sense and maps are multivalued which, for instance, often happens in Convex Geometry as there can be multiple normals at a single boundary point.
 %Can I promote it in any other way?

To conclude the introduction, let us quickly comment on the next chapters. 
Our main result of Section \ref{Section Prop} is variational Lipschitz continuity of the radial Gauss Image, Proposition \ref{continuity of sum}, which is the main ingredient in the mentioned mimic of the Hopf Theorem.  In Section \ref{Section Measure} we establish measure-theoretic notions for multivalued maps, and, in particular, talk about the different examples of possible definitions for almost everywhere equivalence of multivalued maps. We prove Theorem \ref{main} in Section \ref{Section Proof}. The proof for general bodies versus $C^1$ strictly convex bodies differs by usages of Lemma \ref{sub} instead of Lemma \ref{same for strictly}. The Lemma \ref{same for strictly} uses Hopf Theorem and the Lemma \ref{sub} relies on the variational Lipschitz result, Proposition \ref{continuity of sum}, from Section \ref{Section Prop}. In particular, the reader who is only interested in $C^1$ strictly convex result can skip the Sections \ref{Section Prop} and \ref{Section Measure} and address them only if needed. In Section \ref{Section Appl}, we turn to applications of Theorem \ref{main} and obtain Theorem \ref{new additional main} and Theorem \ref{main 2}. 

%TODO Grammar from here.

\section{Preliminaries}\label{Section Prel}
By $\sn$ we denote the unit sphere centered at 0 in $\rn$. We denote $\kn$ to be the set of convex bodies in $\rn$ (compact, convex subsets with nonempty interior in $\rn$). We denote $\kno\subset\kn$ to be those convex bodies that contain the origin in their interiors. If $\partial$ stands next to the set (as $\partial K$) it denotes the topological boundary of the set. If $\partial$ stands next to a convex function (as $\partial f$) it denotes the subdifferential. Given two set $A,B$ By $A\triangle B$ we denote their symmetric difference:
\begin{equation}
	A\triangle B := A\setminus B \cup B \setminus A.
\end{equation}

Given $K\in\kno$, let $x\in\partial K$ be a boundary point. \textit{The normal cone} at $x$ is defined by 
\begin{equation}
  N(K,x)=\{v\in\sn : (y-x)\cdot v\leq 0 \text{ for all } y\in K \}
\end{equation}
 \textit{The radial map} $r_K:\sn\rightarrow\partial K$ of $K$ is defined for $u\in \sn$ by $r_K(u)=ru\in\partial K$, where $r>0$. For $\omega \subset \sn$, \textit{the radial Gauss image} of $\omega$ is defined by
\begin{equation}
\balpha_K(\omega)=\bigcup_{x\in r_K(\omega)}N(K,x)\subset \sn.	
\end{equation}
The radial Gauss image $\balpha_K$ maps sets of $\sn$ to sets of $\sn$. Outside of a spherical Lebesgue measure zero set, the multivalued map $\balpha_K$ is singular valued. It is known that $\balpha_K$ maps Borel measurable sets to Lebesgue measurable sets, closed sets into closed sets. See \cite{S14} for both of these results.  We denote the restriction of $\balpha_K$ to the corresponding singular valued map by $\alpha_K$. For more details, see \cite{GIP}. 

Suppose $\lambda$ is a submeasure on $\sn$ (see Definition \ref{submeasure defenition}) defined on spherical Lebesgue measurable sets and $K\in\kno$. Then $\lambda(K,\cdot)$,  \textit{the Gauss image measure of $\lambda$ via $K$}, is a submeasure defined as the pullback of $\lambda$ via the map $\balpha_K$. That is, for each Borel $\omega\subset\sn$ we set 
\begin{equation}
  \lambda(K,\omega):=\lambda(\balpha_K(\omega)).
\end{equation}
We will say that $\mu$ is a Borel (Lebesgue) measure (or submeasure) if it is a measure (or submeasure) defined on Borel (Lebesgue) sets of $\sn$. In general, given a Lebesgue measure $\lambda$ and $K\in\kno$ it might happen that $\lambda(K,\cdot)$ is not a measure. For example, see \cite{Semenov GIP}. If $\lambda$ is absolutely continuous, $\lambda(K,\cdot)$ is always a Borel measure \cite{GIP}.

\textit{The radial function} $\rho_K:\sn\rightarrow \mathbb R$ is defined by: \begin{equation}
  \rho_K(u)=\max\{a : au\in K\}.
\end{equation}
In this case, $r_K(u)=\rho_K(u)u$. \textit{The support function} is defined by \begin{equation}
  h_K(x)=\max\{x\cdot y : y\in K\}.
\end{equation}
 For $K\in\kno$, we define its \textit{polar body} $K^*$ to be a convex body with support function $h_{K^*}:=\frac 1 {\rho_K}$. We denote by $r_K$ the radius of the largest ball contained in $K$ centered at $o$. Similarly, denote $R_K$ to be the radius of the smallest ball containing $K$ centered at $o$. Clearly, $r_K\leq R_K$. \textit{The support hyperplane} to $K$ with outer unit normal $v\in\sn$ is defined by \begin{equation}
 H_K(v)=\{x:x\cdot v=h_K(v)\}.	
 \end{equation}
We will use the following notation for hyperplanes and half-spaces \begin{equation}
	\begin{split}
		H^-(\alpha,v):= \{x:x\cdot v\leq\alpha\} \\
		H^+(\alpha,v):= \{x:x\cdot v\geq\alpha\} \\
		H(\alpha,v):=\{x:x\cdot v=\alpha\}.
	\end{split}
\end{equation}  By $F(K,v)$ we denote the facet of $K$ in the direction of $v$. That is,

\begin{equation}
	F(K,v):=H_{K}(v)\cap K.
\end{equation}
We define the \textit{the reverse radial Gauss Image Map} of $\omega\subset \sn$ as  

\begin{equation}
\balpha^*_K(\omega)=\{r_K^{-1}(x)\mid x\in \partial K \text{ and } x\in H_K(v) \text{ for some } v\in \omega \}.
\end{equation}
It is shown in \cite{HLYZ16}, that if $K\in\kno$ then for any $\omega\subset \sn $, $\balpha^*_K(\omega)=\balpha_{K^*}(\omega)$. We will use this fact without mentioning it. 

For $\omega \subset S^{n-1}$, define $\text{cone}\,\omega $ as
the {\it cone that $\omega $ generates}, as
\begin{equation}
  \text{cone}\,\omega = \{tu:\text{$t\ge 0$ and $u\in\omega$}\},
\end{equation}
and define $\hat \omega$ to be the {\it restricted cone that $\omega $ generates}, as
\begin{equation}
  \hat \omega =
\{tu : \text{$0\le t \le 1$ and $u\in \omega$}  \}.
\end{equation}
We say that $\omega \subset S^{n-1}$ is {\it spherically convex}, if the cone that $\omega $ generates is a nonempty, proper, convex subset of $\rn$.
%$\gamma \neq \rn$, so that $\omega = \bar \gamma$.
Therefore,
 a spherically convex set in $\sn$ is nonempty and
 %is not the whole sphere.
 %Then a spherically convex set $\omega \subset\sn$
 always contained in a closed hemisphere of $\sn$. We say that  $\omega \subset S^{n-1}$ is {\it spherically convex body} if it is spherically convex set with nonempty interior. Given $\omega\subset S^{n-1}$ contained in a closed hemisphere,
 {\it the polar set} $\omega^*$ is defined by:
\begin{equation}\label{polarset}
\begin{split}	
\omega^* &=
\bigcap_{u\in\omega}\{v\in S^{n-1} : u\cdot v\leq 0 \}. 
%\\
%\omega_\alpha &=
% \bigcup_{u\in \omega} \{v\in \sn : u\cdot v > \cos\alpha \}.
\end{split}
\end{equation}
We note that the polar set is always convex. If $\omega\subset\sn$ is a closed set, we define its {\it outer parallel set} $\omega_\alpha$ to be
\begin{equation}\label{outer parallel set}
	\omega_\alpha =
 \bigcup_{u\in \omega} \{v\in \sn : u\cdot v > \cos\alpha \}.
\end{equation}
By $\bar\omega$ we denote the closure of $\omega$ and by $\interior \omega$ its interior.

We define distance $d$ on $\sn$ to be arc-distance on a sphere. That is, given $u,v\in \sn$: \begin{equation}
  d(u,v)=\lVert u-v\rVert _{\sn}:=\arccos (uv)
\end{equation}
Similarly we define distance between a set $\omega\subset \sn$ and a point $u\in\sn$ as:

\begin{equation}
	d(u,\omega):=\inf\{d(u,v)\mid v\in \omega \}
\end{equation}
We define the Hausdorff distance between $\omega_1,\omega_2\subset \sn$ to be \begin{equation}
  d_H(\omega_1,\omega_2):=\max(\sup_{u\in\omega_1}d(u,\omega_2),\sup_{u\in\omega_2}d(\omega_1,u)) 
\end{equation}	
%We also defined $d_{\sn}(u_1,u_2)$ to be arc length.
Any spherical Lebesgue submeasure $\mu$ on $\sn$ defines a symmetric difference pseudometric $d_\mu$ on a collection of Borel sets on a sphere, where the distance from set $A$ to $B$ is $\mu(A\triangle B)$. We denote this metric by:
\begin{equation}
	d_\mu(A,B):=\mu(A,B).
\end{equation}   

The next two definitions will appear in Section \ref{Section Prop}. Yet, for the reader's convenience, we copy them here. Given $t\in[0,1]$ we define the harmonic mean of $K,L\in\kno$ as \begin{equation}
		K\hat{+}_tL:=((1-t)K^*+tL^*)^*.
	\end{equation}  
By $P:\rn\rightarrow\sn$ we denote the projection of $\rn$ onto $\sn$. Given $t\in[0,1]$ we define the mean of spherically convex sets $\omega_1,\omega_2\in\sn$ contained in same open hemisphere as:
\begin{equation}
  \omega_1\tilde+_t\omega_2:=\{P((1-t)u+tv)\mid u\in\omega_1, v\in\omega_2\}.
\end{equation} 

We use books of Schneider \cite{S14}, Gruber \cite{Gruberbook} and Gardner \cite{G06book} as our references for classical results in convex geometry.

\section{Properties of the Radial Gauss Image Map}\label{Section Prop}

As already mentioned in Section \ref{Section Intro}, given body $K\in\kno$, the radial Gauss image $\balpha_K$ maps sets of $\sn$ into sets of $\sn$ and might be not single-valued. Thus, we need some alternative statement for the continuity of such maps. Recall the continuity of subdifferentials for convex functions: If $f$ is a proper convex function, $x_i\rightarrow x$ and for $x'_i\in\partial f(x_i)$, $x'_i\rightarrow x'$, then $x'\in\partial f(x)$. (See \cite{Rock}). We establish similar continuity properties for the radial Gauss Image map.

\begin{prop}[Continuity of the radial Gauss Image Map]\label{continuity}
	Given body $K\in\kno$, let $u_i$ be a sequence of points on $\sn$ converging to $u\in\sn$. Let $n_i$ be a sequence of points on $\sn$ such that $n_i\in \balpha_K(u_i)$. Then \begin{equation}\label{first part continuity}
		d(n_i,\balpha_K(u))\rightarrow 0
	\end{equation} 
	Moreover,
\begin{equation}
	d_H(\balpha_K(u_\eps),\balpha_K(u))\rightarrow 0 \text{ as } \eps \rightarrow 0.
\end{equation}
\end{prop}
\begin{rmk}
	This is not true if $d$ is substituted with Hausdorff distance.
\end{rmk}
\begin{proof}
Suppose \eqref{first part continuity} is not true. Then, there exist an $\eps>0$ and a subsequence $n_{i_j}$ such that $d(n_{i_j},\balpha_K(u))>\eps$. For notation simplicity, let's assume that this is the original sequence $n_i$. We will obtain a contradiction by constructing a subsequence of $n_i$ which converges to $\balpha_K(u)$. 

	First, we note that since $n_i\in \balpha_K(u_i)$ we have $u_i\in \balpha_{K^*}(n_i)$. Thus, $r_{K*}(n_i)\in F(K^*,u_i)=\partial h(K^*,u_i)$, where $\partial h(K^*,u_i)$ is the subbdiffrential of $h_{K^*}$ at a point $u_i$. For the last equality see \cite{S14}, Theorem 1.7.4. Let $x_i:=r_{K^*}(n_i)\in \partial K^*$. 
	 We have that $x_i\in \partial h(K^*,u_i)$.
	 
Since $\partial K^*$ is compact, for $x_i\in \partial K^*$, there exist subsequence $x_{i_j}$ converging to some $x\in \partial K^*$. Since $x_i\in \partial h(K^*,u_i)$ by continuity of subdiffrential we obtain that $x\in\partial h(K^*,u)=F(K^*,u)$ and, thus, $u\in \balpha_{K^*}(r_{K*}^{-1}(x))$. Therefore, $r_{K*}^{-1}(x)\in \balpha_K(u)$. Now, for subsequence $n_{i_j}=r_{K*}^{-1}(x_{i_j})$ since $x_{i_j}\rightarrow x$, we obtain $n_{i_j}\rightarrow r_{K*}^{-1}(x)$ since the radial map is bi-lipschitz and, thus, we have $d(n_{i_j},\balpha_K(u))\rightarrow 0$. Since for all $i$, $d(n_{i},\balpha_K(u))>\eps$ we obtain a contradiction.

We established \eqref{first part continuity}. Now for the last part, suppose it is not true. Since $u\in u_\eps$ there exist $\delta>0$ and a sequence of $\eps_i$ such that there exist $n_i\in \balpha_K(u_{\eps_i})$ for which $d(n_i,\balpha_K(u))>\delta$. Let $u_i\in u_{\eps_i}$ be such that $n_i\in \balpha_K(u_i)$. Since $u_i\rightarrow u$, from \eqref{first part continuity} we obtain $d(n_i,\balpha_K(u))\rightarrow 0$ which is contradiction.

\end{proof}

Note that for $K\in\kno$ and a Borel set $\omega\subset S^{n-1}$ we don't necessarily have that $\partial\balpha_K(\omega)=\balpha_K(\partial\omega)$. For example, take $K$ to be a regular cube and $\omega=\{v\}$ to be a vector $v$ pointing at the vertex of the cube. Then $\partial\balpha_K(\omega)=\partial\balpha_K(v)\varsubsetneq \balpha_K(v)=\balpha_K(\partial\omega)$. We obtain the following proposition:

\begin{prop}\label{boundary subset}
For any convex body $K\in\kno$, and any Borel set $\omega\subset S^{n-1}$ $\partial\balpha_K(\omega)\subset\balpha_K(\partial\omega)$.
\end{prop}
\begin{proof}
Suppose $\partial\balpha_K(\omega)$ is not empty. Let $n\in \partial\balpha_K(\omega)$. Then, there exists a sequence of vectors $n_i\in \sn$ with $n_i\notin\balpha_K(\omega)$ converging to $n$. Let $u_i$ be any sequence such that $u_i\in\balpha_K^{-1}(n_i)$. Note $u_i\notin\omega$. By compactness, some subsequence of $u_i$, $u_{i_j}$ converge to some $u\in\sn$. Then by continuity of $\balpha_K$, Proposition \ref{continuity}, we obtain that $n\in\balpha_K(u)$, as $\balpha_K(u)$ is a closed set. Since $u_i\notin\omega$ we have that $u\notin\mathring\omega$. 
 
 Suppose $u\notin\partial\omega$ as otherwise we are done. Similarly, we obtain sequence $u'_{i_j}\rightarrow u'$, where $u'_{i_j}\in \omega$ and $n\in \balpha_K(u')$. Suppose as well that $u'\notin\partial\omega$, as otherwise we are done. Now we have two points $u,u'$ sharing the same normal vector $n$ such that $u\in \sn \setminus \bar  \omega$ and $u'\in \mathring \omega$. Since $F(K,n)$ is a convex subset, we obtain that $r_K^{-1}(F(K,n))$ is a spherically convex set containing vectors $u,u'$. Thus, $n$ is normal to any spherical convex combination of $u$ and $u'$. Since $u\in \sn \setminus \bar  \omega$ and $u'\in \mathring \omega$, arc from $u$ to $u'$ passes through $\partial\omega$ and, thus, $n\in \balpha_K(\partial\omega)$.

\end{proof}

We recall the definition of the harmonic mean of two bodies. The definition can be seen in \cite{Firey} or more recently in \cite{Lp Aleks}. We also note  that given vector $u$, $\balpha_{K\hat +_tL}(u)$ is contained in same open hemisphere for all $t$. 

\begin{defi}\label{Defenition of Harmonic Mean Variation}
	Given $t\in[0,1]$ we define the harmonic mean of $K,L\in\kno$ as \begin{equation}
		K\hat{+}_tL:=((1-t)K^*+tL^*)^*.
	\end{equation}  
\end{defi}

It turns out that the radial Gauss Image map behaves in a particularly nice way under the harmonic mean variation. In Proposition \ref{continuity of sum} we are going to establish the Lipschitz properties of this variation. Yet, let us mention that these properties are non needed for the proof of Theorem \ref{main} if one is willing to restrict the theorem to $C^1$ convex bodies with strictly convex duals. If the reader only cares about this case, they can move to the next section immediately. 

We define the mean of convex sets contained in the same open hemisphere as follows:
\begin{defi}
Given $t\in[0,1]$ we define the mean of spherically convex sets $\omega_1,\omega_2\in\sn$ contained in same open hemisphere as:
\begin{equation}
  \omega_1\tilde+_t\omega_2:=\{P((1-t)u+tv)\mid u\in\omega_1, v\in\omega_2\},
\end{equation}
where $P:\rn\rightarrow\sn$ is a projection. 
\end{defi}
In general for two points $u_1,u_2\in\sn$ contained in open hemisphere, $u_1\tilde+_tu_2$
 is the parametrization of the geodesic segment through $u_1$ and $u_2$. Note, that this is not arc-length parametrization.

Next, we establish several identities which show that you can think about $\balpha_{K\hat +_tL}(\omega)$ as a variation from $\balpha_K(\omega)$ to $\balpha_L(\omega)$ along geodesic segments.  

\begin{prop}\label{ess properties} Suppose $K,L\in\kno$. Let $\omega\subset \sn$ be a spherically convex set contained in an open hemisphere and let $u$ be a unit vector. Then the following identities hold:
\begin{enumerate}
	\item $\balpha_{K\hat +_tL}(u)=\{P\big((1-t)v_1+tv_2)\big)\mid v_1\in F(K^*,u), v_2\in F(L^*,u)  \}$ 
	 \item $\balpha_{K\hat +_0L}(\omega)=\balpha_K(\omega)\tilde+_0\balpha_L(\omega)=\balpha_K(\omega)$
	 \item $\balpha_{K\hat +_1L}(\omega)=\balpha_K(\omega)\tilde+_1\balpha_L(\omega)=\balpha_L(\omega)$
	 \item $\bigcup_{t\in(0,1)}\balpha_{K\hat +_tL}(\omega)=\bigcup_{u\in\omega}\bigcup_{t\in(0,1)}\balpha_K(u)\tilde+_t\balpha_L(u)$
\end{enumerate}
\end{prop}
\begin{rmk}
While it is tempting to say $\balpha_{K\hat +_tL}(u)=\balpha_K(u)\tilde+_t\balpha_L(u)$ for any $t\in[0,1]$, it doesn't hold in general. See \eqref{rmk explanation}.
\end{rmk}

\begin{proof}

The second and the third identities are immediate from the definitions. Since for any convex bodies $C,D$: $F(C+D,u)=F(C,u)+F(D,u)$, we obtain the first identity:
\begin{equation}\label{hat explanation}\begin{split}
  \balpha_{K\hat +_tL}(u)=\balpha^*_{(K\hat +_tL)^*}(u)&=\balpha^*_{(1-t)K^*+tL^*}(u)\\&=P\circ F((1-t)K^*+tL^*,u)\\&=P((1-t)F(K^*,u)+tF(L^*,u))\\&=\{P\big((1-t)v_1+tv_2)\big)\mid v_1\in F(K^*,u), v_2\in F(L^*,u)\}.\\
  \end{split}
\end{equation} 

For the last one, we note that $\tilde+$ and $\hat+$ just define different arc-length parameterizations, and therefore images might disagree for some $t$, but they will agree for the union of all $t\in(0,1)$. More formally, from \eqref{hat explanation} we obtain that: 
\begin{equation}\label{comp1}
\begin{split}
\bigcup_{t\in(0,1)}\balpha_{K\hat +_tL}(\omega)&=	\bigcup_{u\in\omega}\bigcup_{t\in(0,1)}\{P\big((1-t)v_1+tv_2)\big)\mid v_1\in F(K^*,u), v_2\in F(L^*,u)  \}\\&=\bigcup_{u\in\omega}\bigcup_{\substack{v_1\in F(K^*,u) \\ v_2\in F(L^*,u)} }\bigcup_{t\in(0,1)}P((1-t)v_1+tv_2)
\end{split}
\end{equation}
Yet, since $\frac{v_1}{\lVert v_1\rVert }$ and $\frac{v_1}{\lVert v_1\rVert }$ are contained in the same open hemisphere we have: \begin{equation}\label{rmk explanation}
  \bigcup_{t\in(0,1)}P((1-t)v_1+tv_2)=\bigcup_{t\in(0,1)}P((1-t)\frac{v_1}{\lVert v_1\rVert }+t\frac{v_2}{\lVert v_2\rVert})
\end{equation}
Again, we note that the above is not necessarily true for specific $t$ but it is true for the union of $t\in(0,1)$. Therefore, \eqref{comp1} implies that \begin{equation}
 	\bigcup_{t\in(0,1)}\balpha_{K\hat +_tL}(\omega)=\bigcup_{u\in\omega}\bigcup_{\substack{v'_1\in \balpha^*_{K^*}(u) \\ v'_2\in \balpha^*_{L^*}(u) }}\bigcup_{t\in(0,1)}P((1-t)v'_1+tv'_2)
 \end{equation}
where $v'_1=\frac{v_1}{\lVert v_1\rVert }$ and $v'_2=\frac{v_2}{\lVert v_2\rVert }$. From the last equation by recalling the definition of $\tilde+$ we obtain the third statement. 
\end{proof}

We are now going to establish Lipschitz properties for $\tilde+$ and $\hat+$.

\begin{lem}\label{lip harm}
	Given $v\in\sn$ and $\alpha>0$, suppose we are given $u_1,u_2$ not equal to zero such that $\frac{u_1}{\lVert u_1\rVert },\frac{u_2}{\lVert u_2\rVert }\in v_{\frac \pi 2 - \alpha}$. Then $g(t):[0,1]\rightarrow  P((1-t)u_1+tu_2)$ is Lipschitz continuous  with respect to distance on the sphere with constant $L=\frac {2}{\sin(\alpha)}\max ({\frac {\lVert u_1\rVert }{\lVert u_2\rVert }},{\frac {\lVert u_2\rVert }{\lVert u_1\lVert }})$. That is,
	\begin{equation}
		d(g(t_1),g(t_2))\leq L \lVert t_1-t_2\rVert
	\end{equation}
\end{lem}
\begin{proof}
If $\frac{u_1}{\lVert u_1\rVert }=\frac{u_2}{\lVert u_2\rVert }$ then the statement trivially holds as $g(t)$ is constant. Suppose not. Without loss of generality assume that $\lVert u_1\rVert \leq\lVert u_2\rVert $.  Let $O$ be center of coordinates and $v(t):=(1-t)u_1+tu_2$. Note that $g(t)=P(v(t))$. We consider triangle in $\rn$ with vertices $O$, $u_1$ and $v(t)$. Let $\theta$ be an angle opposite to side $[\text{O},v(t)]$ an $\gamma(t)$ be an angle in this triangle opposite to the side $[u_1,v(t)]$. 

Since $u_1,u_2\in v_{\frac \pi 2 - \alpha}$ we have that as $t$ increases from 0 to 1, $\gamma$ increases from 0 to some angle $\gamma(1)$ which is less than $\pi - 2\alpha$. Since $\lVert u_1\rVert \leq\lVert u_2\rVert $ and $\gamma(1)<\pi - 2\alpha$ we obtain that $\theta>\alpha$. Clearly, $\theta<\pi$. We note that $\gamma(t)=d(g(0),g(t))$. We also note that $\lVert u_1-v(t)\rVert =t\lVert u_1-u_2\rVert $. 

Now, from the law of sines for the triangle $[O,u_1,v(t)]$ with $0<t\leq 1$ we obtain: 
\begin{equation}\begin{split}
  \frac {\lVert u_1-v(t)\rVert }{\sin(\gamma(t))}&=\frac{\lVert u_1\rVert }{\sin(\pi-\theta-\gamma(t))} \Leftrightarrow \\
  {t\lVert u_1-u_2\rVert }&=\lVert u_1\rVert \frac{{\sin(\gamma(t))}}{\sin(\pi-\theta-\gamma(t))} \Leftrightarrow \\
  {t\lVert u_1-u_2\rVert }&=\lVert u_1\rVert \frac{{\sin(\gamma(t))}}{\sin(\theta+\gamma(t))}
\end{split}
\end{equation}
Note that the last equation trivially holds for $t=0$, so it holds for all $t\in[0,1]$.

Let $f(t,\gamma):={t\lVert u_1-u_2\rVert _{\rn}}-\lVert u_1\rVert \frac{{\sin(\gamma)}}{\sin(\theta+\gamma)}$. Since $\theta$ and $\gamma(1)$ are two angles in non-degenerate triangle, $0<\gamma(t)+\theta<\pi$ for $t\in[0,1]$, and, hence, $f(t,\gamma)$ is continuously differentiable around $(t,\gamma(t))$ for any $t\in[0,1]$. Moreover for $(t,\gamma(t))$ with $t\in[0,1]$:

\begin{equation}
\begin{split}
	\frac {\partial f}{\partial \gamma} &= - \lVert u_1\rVert \frac{\cos(\gamma)\sin(\theta+\gamma)+\sin(\gamma)\cos(\theta+\gamma)}{\sin^2(\theta+\gamma)}\\&=- \lVert u_1\rVert \frac{\sin(\theta+\gamma-\gamma)}{\sin^2(\theta+\gamma)} \\&= - \lVert u_1\rVert \frac{\sin(\theta)}{\sin^2(\theta+\gamma)}\neq 0 \\
	\frac {\partial f}{\partial t} &= \lVert u_1-u_2\rVert _{\rn}
	\end{split}
\end{equation} 
Therefore, using implicit function theorem and $\lVert u_1\rVert \leq\lVert u_2\rVert $ we obtain:
\begin{equation}\label{diff gamma}
\begin{split}
	\gamma'(t)&=\frac{\lVert u_1-u_2\rVert _{\rn}}{\lVert u_1\rVert }\frac{\sin^2(\theta+\gamma(t))}{\sin(\theta)}\leq  \frac {2\lVert u_2\rVert }{\lVert u_1\rVert }\frac{\sin^2(\theta+\gamma(t))}{\sin(\theta)} \Leftrightarrow \\ 
	\gamma'(t)&\leq \frac {2\sin^2(\theta+\gamma(t))}{\sin(\theta)}\max ({\frac {\lVert u_1\rVert }{\lVert u_2\rVert }},{\frac {\lVert u_2\rVert }{\lVert u_1\rVert }})
	\end{split}
\end{equation}
Note that the last equation similarly holds for $\lVert u_1\rVert \geq\lVert u_2\rVert $.

%We implicitly differentiate $\gamma(t)$ for $t\in[0,1]$ with respect to $t$. (Add justification). Differentiating we obtain 
%\begin{equation}\label{7.12}
%\begin{split}
%  ||u_1-u_2||_{\rn}&=||u_1||_{\rn}\frac{\cos(\gamma)\gamma'\sin(\theta+\gamma)+\sin(\gamma)\cos(\theta+\gamma)\gamma'}{\sin^2(\theta+\gamma)} \Leftrightarrow \\
%  ||u_1-u_2||_{\rn}&=||u_1||_{\rn}\gamma'\frac{\sin(\theta+\gamma-\gamma)}{\sin^2(\theta+\gamma)} 
%\end{split}
%\end{equation}

We recall that $\alpha<\theta<\pi$ and $0\leq\gamma<\pi - 2\alpha$. Now we analyze two cases. First, assume that $\theta\leq\pi-\alpha$. Then we get that 
\begin{equation}\label{diff gamma 2}
	\frac {2\sin^2(\theta+\gamma(t))}{\sin(\theta)}\leq \frac 2 {\sin(\alpha)}
\end{equation}
and the statement follows from \eqref{diff gamma} and \eqref{diff gamma 2} combined. Suppose now, that $\theta>\pi-\alpha$. In particular, $\theta>\frac \pi 2$. This, combined with $0<\gamma(t)+\theta<\pi$ implies that 
\begin{equation}
  \frac{\sin(\theta+\gamma)}{\sin(\theta)}\leq 1
\end{equation}
Therefore,  \eqref{diff gamma} for this case implies the the following:
\begin{equation}
\gamma'=\frac {2\sin^2(\theta+\gamma(t))}{\sin(\theta)}\max ({\frac {\lVert u_1\rVert }{\lVert u_2\rVert }},{\frac {\lVert u_2\rVert }{\lVert u_1\rVert }})
\leq2\max ({\frac {\lVert u_1\rVert }{\lVert u_2\rVert }},{\frac {\lVert u_2\rVert }{\lVert u_1\rVert }})<\frac {2}{\sin(\alpha)}\max ({\frac {\lVert u_1\rVert }{\lVert u_2\rVert }},{\frac {\lVert u_2\rVert }{\lVert u_1\rVert }})
\end{equation}

\end{proof}

The next proposition is essential.

\begin{prop}\label{continuity of sum}
Given any $\omega\subset S^{n-1}$, $\balpha_{K\hat{+}_tL}(\omega):[0,1]\rightarrow \sn$ is Lipschitz continuous map from $t$ to sets on sphere equipped with Hausdorff distance $d_H$:\begin{equation}\label{3.160}
	d_H(\balpha_{K\hat{+}_{t_1}L}(\omega),\balpha_{K\hat{+}_{t_2}L}(\omega))\leq2 \max (\frac {R_K} {r_K}, \frac {R_L} {r_L})\max (\frac {R_K} {r_L}, \frac {R_L} {r_K})|t_1-t_2|.
\end{equation} 
\end{prop}
\begin{proof}
Suppose $0\leq t_1\leq t_2 \leq1$. Let $b$ be constants satisfying: \begin{equation}
	0<b<\min \{\frac {r_K}{R_K},\frac {r_L}{R_L}\}
	\end{equation}By properties of Hausdorff distance:
\begin{equation}\label{3.12}
  \begin{split}
  	&d_H(\balpha_{K\hat{+}_{t_1} L}(\omega),\balpha_{K\hat{+}_{t_2}L}(\omega))\leq \\&d_H(\bigcup_{u\in\omega}\balpha_{K\hat{+}_{t_1} L}(u),\bigcup_{u\in\omega}\balpha_{K\hat{+}_{t_2} L}(u))\leq \\& \sup_{u\in\omega}d_H(\balpha_{K\hat{+}_{t_1} L}(u),\balpha_{K\hat{+}_{t_2} L}(u))
  \end{split}
\end{equation}
Fix $u\in \sn$ we are going to show that $\balpha_{K\hat{+}_{t} L}(u)$ is Lipschitz with same constant as in equation \eqref{3.160}. Then the above equation \eqref{3.12} will imply the claim. From Proposition \ref{ess properties} we have:
\begin{equation}\label{3.14}
\begin{split}
  \balpha_{K\hat{+}_{t} L}(u)=\bigcup_{\substack{v_1\in F(K^*,u) \\ v_2\in F(L^*,u)}}P\big((1-t)v_1+tv_2\big)
  \end{split}
\end{equation}
Using Lemma 3.5 in \cite{Semenov GIP} for $u\in\omega$: \begin{equation}
  \balpha_K(u)\subset u_{\arccos b}=u_{\frac \pi 2 - (\frac \pi 2 - \arccos b)}
  \end{equation}
Therefore, for all $v_1\in F(K^*,u)$ we have that $v_1\in u_{\frac \pi 2 - (\frac \pi 2 - \arccos b)}$. Similarly, $v_2\in u_{\frac \pi 2 - (\frac \pi 2 - \arccos b)}$ for all $v_2\in F(L^*,u)$. Therefore we can invoke Proposition \ref{lip harm}, and combining it with \eqref{3.14} we obtain that the following holds for equation \eqref{3.12}: 
\begin{equation}
  \begin{split}
  &\sup_{\phantom{......}u\in\omega\phantom{.......}}d_H(\balpha_{K\hat{+}_{t_1} L}(u),\balpha_{K\hat{+}_{t_2} L}(u))\leq\\&\sup_{\substack{u\in\omega \\ v_1\in F(K^*,u) \\ v_2\in F(L^*,u)}}d_H(P\big((1-t_1)v_1+t_1v_2\big),P\big((1-t_2)v_1+t_2v_2\big))\leq\\
   & \sup_{\substack{u\in\omega \\ v_1\in F(K^*,u) \\ v_2\in F(L^*,u)}} \frac 2 {\sin (\frac \pi 2 -\arccos b)} \max(\frac {\lVert v_1\rVert }{\lVert v_2\rVert },\frac {\lVert v_2\rVert }{\lVert v_1\rVert })|t_2-t_1| \leq  \\ & \phantom{.....}\frac 2 b \max(\frac {R_K} {r_L}, \frac {R_L} {r_K}) |t_2-t_1| \leq  \\& \phantom{.....}
2 \max (\frac {R_K} {r_K}, \frac {R_L} {r_L})\max (\frac {R_K} {r_L}, \frac {R_L} {r_K})|t_2-t_1|
  \end{split}
\end{equation}

\end{proof}

%Grammar from here.

\section{Multivalued Maps and Related Measure Theory}\label{Section Measure}
 In this section, we are going to elaborate on measure-theoretic statements with regard to the Theorem \ref{main}. First, let us address the almost everywhere equality condition for the multi-valued maps. If we are given two measurable functions $f,g : \mathbb{R} \rightarrow \mathbb{R}$, then the usual definition of $f$ being equal to $g$ almost everywhere with respect to, say, Lebesgue measure $\lambda$ is given by 
\begin{equation}\label{standard a.e.}
	\lambda(\{x\mid f(x)\neq g(x)\})=0
\end{equation}
We claim that this is the same as a saying 
\begin{equation}\label{new a.e}
	\forall \omega \text{ Borel }\lambda(f^{-1}(\omega)\triangle g^{-1}(\omega))=0, 
\end{equation}
where $\triangle$ denotes the symmetric difference between two sets. Clearly \eqref{standard a.e.} implies \eqref{new a.e} as 
\begin{equation}
	f^{-1}(\omega)\triangle g^{-1}(\omega)\subset \{x\mid f(x)\neq g(x)\}
\end{equation}
For the reverse direction given $\eps>0$ write $\mathbb R$ as the union of intervals $I_j$ with length $\eps$. Then
\begin{equation}\begin{split}
	\lambda(\{x\mid |f(x)-g(x)|>\eps \})= \\
	\lambda\Big (\bigcup_j \big (\{x\mid |f(x)-g(x)|>\eps \}\cap f^{-1}(I_j)\big )\Big )=\\ \lambda\big (\bigcup_j \big \{x\mid |f(x)-g(x)|>\eps, f(x)\in I_j \} \big )\leq \\
	  \lambda\Big  (\bigcup_j \big \{x\mid f(x)\in I_j, g(x)\notin I_j \}  \Big  ) \leq \\
	\sum_i \lambda(f^{-1}(I_j)\triangle g^{-1}(I_j)) = 0
\end{split}
\end{equation}
Since this is true for arbitrary $\eps$ we obtain \eqref{standard a.e.}.

However, when $f$ and $g$ are multivalued functions, such as, for example, radial Gauss Image maps, equation \eqref{standard a.e.} turns out to be much stronger than \eqref{new a.e}. To see this, consider two functions \begin{equation}\begin{split}
	f(x)&= (x-1,x+1) \\
	g(x)&=[x-1,x+1]
\end{split}
\end{equation}
Then, $\{x\mid f(x)\neq g(x)\} = \mathbb R$, so \eqref{standard a.e.} clearly doesn't hold. Let's show that these two functions satisfy \eqref{new a.e}.  Given a set $\omega$:
\begin{equation}\label{functions not equiv}
	f^{-1}(\omega)\triangle g^{-1}(\omega)=\Big(\bigcup_{x\in \omega}[x-1,x+1] \Big )\setminus \Big(\bigcup_{x\in \omega}(x-1,x+1) \Big )
\end{equation}
Define a set
\begin{equation}
	I=\bigcup_{x\in \omega}[x-1,x+1] 
\end{equation}
Let $n\in \mathbb Z$, then either $I\cap [n,n+1]=[n,n+1]$ or $I\cap [n,n+1]=I_1\cup I_2$ where $I_1$ is an interval which is either empty or it contains $n$ and $I_2$ is an interval which is either empty or it contains $n+1$. Both of these intervals might be independently half-open or closed. In any of these cases, it is easy to see that one can choose at most countable subset $\omega_n$ from $\omega$ so that

\begin{equation}
	I\cap [n,n+1]\subset \bigcup_{x\in \omega_n}[x-1,x+1].
\end{equation}
Let $\omega_c=\bigcup_{n\in \mathbb Z} \omega_n$. We obtain \begin{equation}
	I=\bigcup_{x\in \omega_c}[x-1,x+1].
\end{equation}
Therefore, returning to \eqref{functions not equiv} we obtain:
\begin{equation}\begin{split}
	f^{-1}(\omega)\triangle g^{-1}(\omega)= \\ \Big(\bigcup_{x\in \omega_c}[x-1,x+1] \Big )\setminus \Big(\bigcup_{x\in \omega}(x-1,x+1) \Big )\subset \\
	\bigcup_{x\in \omega_c}\{x-1,x+1\} 
\end{split}
\end{equation}
which is at the most countable set. Therefore, for any set $\omega$,
\begin{equation}
	\lambda(f^{-1}(\omega)\triangle g^{-1}(\omega))=0,
\end{equation}
thus satisfying \eqref{new a.e} for any Borel set $\omega$.
%First of all notice that clearly $\mathbb R\in \mathcal A$. If $\omega \in \mathcal A$ then \begin{equation}
%\begin{split}
%	f^{-1}(\mathbb R \setminus \omega)\triangle g^{-1}(\mathbb R \setminus \omega)= \\\big (f^{-1}(\mathbb R \setminus \omega)\setminus g^{-1}(\mathbb R \setminus \omega)\big ) \cup \big (g^{-1}(\mathbb R \setminus \omega)\setminus f^{-1}(\mathbb R \setminus \omega)\big ) =\\ \{x\mid f(x)\not\subset \omega, g(x)\subset  \omega\}\cup \{x\mid g(x)\not\subset \omega, f(x)\subset  \omega\}= \\ \big(g^{-1}(\omega)\setminus f^{-1}(\omega)\big )\cup \big ( f^{-1}(\omega)\setminus g^{-1}(\omega)\big )=\\
%	g^{-1}(\omega)\triangle f^{-1}(\omega).
%\end{split}
%\end{equation}

Guided by this example we are going to introduce the Definition \ref{Almost everywhere equal definition}, which seems to be a proper generalization for the concept of equivalence of almost everywhere when instead of single-valued functions we consider multivalued functions. First, we start with some basic definitions.

\begin{defi}\label{submeasure defenition}
	A spherical submeasure $\mu:\mathcal B \rightarrow [0,+\infty)$ defined on a $\sigma$-algebra $\mathcal B$ of subsets of $\sn$ is a function which satisfies the following:
	\begin{itemize}
		\item $\mu(\emps)=0$
		\item If $A,B\in \mathcal B$ and $A \subset B$, then $\mu(A)\leq \mu(B)$
		\item If $A_1,A_2,\ldots\in\mathcal B$ then $\mu(\bigcup_{i}A_i)\leq \sum_{i}\mu(A_i)$.
	\end{itemize}
\end{defi}

We say that a multi-valued map $f:\sn\rightarrow\sn$ is measurable with respect to spherical submeasure $\lambda$ defined on $\mathcal B$ if for any Borel $\omega\subset \sn$ we have that set $f^{-1}(\omega):=\{x\mid f(x)\cap \omega \neq \emps\}\in \mathcal B$. With this setting given $K\in\kno$, the map $\balpha_{K^*}$ is measurable with respect to any $\lambda$ defined on the Lebesgue measurable sets, as $\balpha_{K^*}^{-1}(\cdot)=\balpha_{K}(\cdot)$ maps Borel measurable sets to Lebesgue measurable sets.

\begin{defi}\label{Almost everywhere equal defenition}
	Given spherical submeasure $\lambda$ defined on $\sigma$-algebra $\mathcal B$, multivalued measurable maps $f$ and $g$ from $\sn$ into $\sn$ are equal almost everywhere if for any $\omega$ Borel,
	\begin{equation}\label{to prove}
		\lambda (f^{-1}(\omega)\triangle g^{-1}(\omega))=0,
	\end{equation}
\end{defi}

The equation \eqref{to prove} is precisely the statement we are going to prove for Theorem \ref{main}:

 \begin{thm}
 Let $K,L\in \kno$. Suppose $\lambda(K,\cdot)=\lambda(L,\cdot)$ are finite Borel measures for a spherical Lebesgue submeasure $\lambda$. Then for any Borel set $\omega\subset S^{n-1}$,  \begin{equation}\label{to prove 2}
 \lambda(\balpha_K(\omega)\triangle\balpha_L(\omega))=0.
 \end{equation}
 \end{thm}
 
 It turns out that one can other versions of \eqref{to prove 2} when combined with the continuity of the radial Gauss Image maps, Proposition \ref{continuity}. For this, see Lemma \ref{nonempty} and Proposition \ref{common continuity}. 
 
To conclude the measure-theoretic discussion let us now address the measure theory assumptions in Theorem \ref{main}. Given a spherical submeasure $\lambda$ defined on Lebesgue measurable sets and the body $K\in\kno$, the radial Gauss Image measure of $\lambda$ via $K$, $\lambda(K,\cdot)$ is a spherical Borel submeasure, see \cite{GIP}. If $\lambda$ is assumed to be measure instead of submeasure, then $\lambda(K,\cdot)$ becomes a Borel measure as well. This follows from countable additivity of $\lambda$ and the following equality for any countable collection $\omega_i$ of Borel sets :

\begin{equation}
	\bigcup_i\balpha_K(\omega_i)=\balpha_K(\bigcup_i\omega_i).
\end{equation}

As one will soon see in the proof of Theorem \ref{main}, there are several places where we use different properties of $\lambda$ and $\lambda(K,\cdot)$. For example, the regularity of measure $\lambda(K,\cdot)$ is used in Lemma \ref{open} and the continuity of measure $\lambda(K,\cdot)$ is used in Lemma \ref{gamma}.  One can also ask the question whether Theorem \ref{main} holds for submeasure $\lambda(K,\cdot)$. We use the assumption that $\lambda(K,\cdot)$ is a measure when we invoke Lemma 6.3 from \cite{GIP}. Finiteness of measure $\lambda(K,\cdot)$ is never used, other than that finite Borel measures on $\sn$ are automatically regular. 

In general, we didn't try to make the assumption in Theorem \ref{main} as weak as possible and we kept them the way they are stated right now, as all of the convex geometric measures or submeasures (Aleksandrov's integral curvature \cite{Aleks}, Surface area measure of Aleksandrov-Fenchel-Jessen \cite{Aleks0}, Dual curvature measures \cite{HLYZ16}) associated with the Gauss Image Problem, work with Theorem \ref{main} assumptions.

%TODO check the grammar here. 

\section{Proof of the Main Theorem}\label{Section Proof}

Any spherical Lebesgue submeasure $\mu$ defines symmetric difference pseudometric  $d_\mu$ on a collection of Borel sets on a sphere, where the distance from set $A$ to $B$ is $\mu(A\triangle B)$. In this notation Theorem \ref{main} now becomes a statement about maps  $\balpha_K$, $\balpha_L$ being equal as maps to symmetric difference pseudometric spaces. We obtain the following:

\begin{lem}\label{open}
	Let $K,L\in \kno$. Suppose $\lambda(K,\cdot)=\lambda(L,\cdot)$ are finite Borel measures for a spherical Lebesgue submeasure $\lambda$. If for some $\omega$ Borel $\lambda(\balpha_K(\omega)\triangle\balpha_L(\omega))>0$, then there exist an open set $\omega'$ such that $\lambda(\balpha_K(\omega')\triangle\balpha_L(\omega'))>0$ 
\end{lem}
\begin{proof}
Suppose for some $\omega$ Borel $d_\lambda(\balpha_K(\omega),\balpha_L(\omega))=\lambda(\balpha_K(\omega)\triangle\balpha_L(\omega))=\eps>0$. Note that any finite Borel measure on sphere is automatically a regular Borel measure, because, for example, a sphere is locally compact, separable metric space. Since $\lambda(K,\cdot)$ is a regular Borel measure there exist $\omega'_K$ open set such that $\omega\subset\omega'_K$ and $\lambda(K,\omega'_K\setminus\omega)$ is less than $\eps/4$. Similarly, define $\omega'_L$. Let $\omega'=\omega'_K\cap\omega'_L$. Therefore, \begin{equation}\label{triangle ineq}\begin{split}
	d_\lambda(\balpha_K(\omega'),\balpha_K(\omega))= \\ \lambda(\balpha_K(\omega')\triangle\balpha_K(\omega))=\\ \lambda(\balpha_K(\omega')\setminus\balpha_K(\omega))\leq \\ \lambda(\balpha_K(\omega'\setminus\omega))\leq \\ \lambda(\balpha_K(\omega'_K\setminus\omega))< \\ \eps/4.
	\end{split}
\end{equation} Similarly, $d_\lambda(\balpha_L(\omega'),\balpha_L(\omega))<\eps/4$. Then, by triangle inequality, we get $d_\lambda(\balpha_K(\omega'),\balpha_L(\omega'))>\eps/4$, where $\omega'$ is some open Borel set.	
\end{proof}

%TODO check this lemma, and how to write it.

Lemma 6.3 in \cite{GIP} essentially states that given any measure $\mu$ defined on a sphere almost every, in Haar sense, spherical rectangle has $\mu$ zero measure boundary. So in particular there should exist semiring of rectangles with $\mu$ zero measure boundary which defines Borel sets as ring closure. We are not going to prove this or make this more precise but instead stick with the statement of Lemma 6.3 in \cite{GIP}. Noticing that the proof of the Lemma 6.3 in \cite{GIP} goes through when considering two measures at the same time we obtain the following statement. 

\begin{lem}\label{grid}
Let $K,L\in \kno$. Suppose $\lambda(K,\cdot)=\lambda(L,\cdot)$ are finite Borel measures. For any $\eps>0$ there exist a constant $m(\lambda(K,\cdot),\lambda(L,\cdot),\eps)$ and a collection of compact spherically convex bodies $\Omega_\eps:=\{\omega^\eps_1,...,\omega^\eps_m\}$ satisfying:
\begin{enumerate}
	\item their union is $S^{n-1}$
	\item diameter of each set is less than $\eps$, that is for any $x,y\in \omega_j^\eps $ we have $d(x,y)<\eps $.
	\item any of their pairwise intersection have zero $\lambda(K,\cdot)$ and zero $\lambda(L,\cdot)$ measure.
	\item $\lambda(K,\partial\omega^\eps_j)=\lambda(L,\partial\omega^\eps_j)=0$.
\end{enumerate} 
\end{lem}

Using Lemma \ref{grid}, we obtain the following proposition: 

\begin{lem}\label{gamma}
Let $K,L\in \kno$. Suppose $\lambda(K,\cdot)=\lambda(L,\cdot)$ are finite Borel measures for a spherical Lebesgue submeasure $\lambda$. If for some Borel set $\omega$, $\lambda(\balpha_K(\omega)\triangle\balpha_L(\omega))>0$, then  there exist a compact spherically convex body $\gamma$, such that \begin{enumerate}
	\item Set $\balpha_K(\gamma)\cup\balpha_L(\gamma)$ is contained in open hemisphere.
	\item $\lambda(\balpha_K(\gamma)\triangle\balpha_L(\gamma))>0$
	\item $\lambda(\balpha_K(\partial\gamma))=\lambda(\balpha_L(\partial\gamma))=0$
	\item
	$\lambda \Big (\balpha_K(\gamma)\triangle\balpha_L(\gamma) \setminus \big (\balpha_K(\partial\gamma)\cup\balpha_L(\partial\gamma)\big )\Big )>0$
	\end{enumerate}

\end{lem}

\begin{proof}
	By Lemma \ref{open} we may assume that $\omega$ is an open set. Suppose $d_\lambda(\balpha_K(\omega),\balpha_L(\omega))=\lambda(\balpha_K(\omega)\triangle\balpha_L(\omega))=\eps>0$. Given $\delta>0$, let $\omega_\delta$ be defined as union of all sets from $\Omega_\delta$ which are contained in $\omega$. Clearly, $\omega_\delta$ is closed and $\omega_\delta\subset\omega$. Define,

\begin{equation}\label{omega n construction}
  \gamma_n:=\bigcup_{k=1}^n\omega_{\frac 1 {k}}.
\end{equation}
Note that $\gamma_n$ is a closed set. Construct a sequence of sets $\theta_n=\omega\setminus\gamma_n$, which is a sequence of nested open sets. First, we claim that $\theta_n$ is a a sequence of nested sets such that \begin{equation}
	\bigcap_{n=1}^\infty \theta_n = \emps.
\end{equation} 
Suppose not, that is suppose some $u\in \bigcap_{n=1}^\infty \theta_n$. Since $\omega$ is an open set, there exist an $r>0$ such that $u\in B_r(u)\subset \omega$. But then, for any natural $k>\frac 2 r $ there exists a set in $\Omega_{\frac 1 k}$ covering point $u$ and contained in $B_r(u)$ by first two properties from Lemma \ref{grid}. This, combined with $B_r(u)\subset \omega$ implies that $u\in  \omega_{\frac 1 k} \in \gamma_k$. Therefore, $u\notin \theta_k$ contradicting $u\in \bigcap_{n=1}^\infty \theta_n$.

Since $\theta_n=\omega\setminus\gamma_n$ is a sequence of nested sets converging to zero by continuity of measure \begin{equation}\begin{split}
	\lambda(K,\omega\setminus\gamma_n)&\longrightarrow0\\
	\lambda(L,\omega\setminus\gamma_n)&\longrightarrow0.
\end{split}
\end{equation} 
By argument identical to Equation \eqref{triangle ineq} We obtain that there exists $n$ such that
\begin{equation}
\begin{split}
	d_\lambda(\balpha_K(\omega),\balpha_K(\gamma_n))<\frac \eps 4\\
	d_\lambda(\balpha_L(\omega),\balpha_L(\gamma_n))<\frac \eps 4 
\end{split}
\end{equation} 
Therefore, combining this with $d_\lambda(\balpha_K(\omega),\balpha_L(\omega))=\eps$ and using triangle inequality we obtain: 
\begin{equation}\label{to combine}
	d_\lambda(\balpha_K(\gamma_n),\balpha_L(\gamma_n))>\eps/4.
\end{equation}
 Note that, for any $\mathcal{I}$, an arbitrary non-empty index set, and any collections of sets $A_\alpha, B_\alpha$ with this indexing set we have:\begin{equation}\label{big bracket}
  \big (\bigcup_{\alpha\in\mathcal{I}}A_\alpha \big )\triangle \big (\bigcup_{\alpha\in\mathcal{I}}B_\alpha\big )\subseteq\bigcup_{\alpha\in\mathcal{I}}\left(A_\alpha\triangle B_\alpha\right).
\end{equation}
We also note that for $C=K$ or $C=L$, \begin{equation}\label{obvious one}
	\balpha_C\big ( \bigcup_{\alpha\in\mathcal{I}} A_\alpha\big  )= \bigcup_{\alpha\in\mathcal{I}} \balpha_C(A_\alpha)
\end{equation} Therefore, combining equations \eqref{to combine}, \eqref{big bracket}, \eqref{obvious one} and recalling definition of $\gamma_n$ we obtain that there exist some $\omega^\delta_j$ for which $d_\lambda(\balpha_K(\omega^\delta_j),\balpha_L(\omega^\delta_j))>0$. 

Let $\gamma:=\omega^\delta_j$. Since $d_\lambda(\balpha_K(\omega^\delta_j),\balpha_L(\omega^\delta_j))>0$ we obtain that the required second property for $\gamma$ is satisfied. Since $\gamma=\omega^\delta_j$, the third property follows from Lemma \ref{grid}. The last property follows from the second and the third statement. We are only left to show the first statement. Notice that the proof goes unchanged if we redefine $\gamma_n$ in \ref{omega n construction} to start from any $k=v\in \mathbb N$,
\begin{equation}
	 \gamma_n:=\bigcup_{k=v}^n\omega_{\frac 1 {k}}.
\end{equation}
Thus, we can insure the diameter of $\gamma$ is smaller than any given $\eps$. Combining this with Lemma 3.5 from \cite{Semenov GIP} it is not hard to see that one can find $\eps$, so that $\balpha_K(\gamma)\cup\balpha_L(\gamma)$ is contained in an open hemisphere which completes the proof.
\end{proof}

We now would like to show that for set $\gamma$ in Lemma \ref{gamma} the following equation holds:

\begin{equation}\label{super cool equation}
	\balpha_K(\gamma)\triangle\balpha_L(\gamma) \setminus \big (\balpha_K(\partial\gamma)\cup\balpha_L(\partial\gamma)\big )\subset  \bigcup_{0< t<1}\balpha_{K\hat +_tL}(\partial\gamma)
\end{equation}
This equation is the core element for the proof of Theorem \ref{main}. While this may be complicated notationally, this essentially means that as we change $K$ to $L$ by $K\hat +_tL$, normals cones at $\partial\gamma$ change continuously, and swipe out through symmetric differences of $\balpha_K(\gamma)$ and $\balpha_L(\gamma)$. 

We establish two different proofs for Equation \eqref{super cool equation}. The first proof is for $C^1$ strictly convex bodies which beautifully follows from Hopf's Theorem and doesn't use Proposition \ref{continuity of sum} about Lipschitz continuity of $\balpha_{K\hat+_tL}(\omega)$. Unfortunately, we were not able to naturally extend this proof to $\kno$ because of the complicated structure of the radial Gauss image map. When $K,L$ are $C^1$ strictly convex bodies, $\balpha_{K\hat+_tL}:\sn\rightarrow \sn $ defines a homotopy of homeomorphisms $\balpha_K$ and $\balpha_L$. This fails for more general $K,L$ since maps $\balpha_K$ and $\balpha_L$ become multivalued, and not, necessarily injective, thus, making standard topological machinery inapplicable in the most general case. We find a different proof for the most general case mimicking the Hopf Theorem and using Proposition \ref{continuity of sum}. Yet, we still include the proof for the $C^1$ strictly convex case as it contains the main ideas, intuition and motivation. Let us note that it would be very interesting to see if somebody can obtain \eqref{super cool equation} just from topological ideas without using notion of distance, as it was done here for $C^1$ strictly convex bodies.

We start with a preliminary proposition for $C^1$ strictly convex bodies.

\begin{prop}\label{homemomorphism}
	Suppose $K\in\kno$ is $C^1$ strictly convex body. Then $\balpha_K:\sn\rightarrow \sn$ is a homeomorphism. 
\end{prop}
	\begin{proof}
		Since $K$ is $C^1$ strictly convex body, $\balpha_K$ is a continuous  bijective map. Therefore, by invariance of the domain $\balpha_K$ is a homeomorphism. 
	\end{proof}
When $K\in\kno$ is $C^1$ strictly convex body we have an improvement for Proposition \ref{boundary subset}. 

\begin{prop}\label{boundary for C^1}
	Suppose $K\in\kno$ is $C^1$ strictly convex body. Then for any $\omega\subset\sn$, \begin{equation}
		\balpha_K(\gamma)\setminus \balpha_K(\partial \gamma)= \balpha_K(\interior \gamma)
	\end{equation}
\end{prop}
\begin{proof}
This immediately follows from Proposition \ref{homemomorphism}.
%	If $K\in\kno$ has strictly convex dual, then $\balpha_K^*$ is a continuous map from $\sn$ into $\sn$ instead of multivalued map. Given any $K\in\kno$ we have \begin{equation}\label{inclusion for strict}
%		\balpha_K(\gamma)\setminus \balpha_K(\partial \gamma)\subset \balpha_K(\interior \gamma).
%	\end{equation} Suppose now, $K$ has strictly convex dual, then we have strict inclusion only if there exist $u\in\interior\gamma$ such that $\balpha_K(u)\in \balpha_K(\partial \gamma)$. Therefore, $\balpha_K^*(\balpha_K(u))$ has at least two images, which contradicts the fact that $\balpha_K^*$ is single-valued map.
\end{proof}

We are ready to show Equation \eqref{super cool equation} for $C^1$ strictly convex bodies. 

\begin{prop}\label{same for strictly}
	Let $K,L\in\kno$ be $C^1$ strictly convex bodies. Suppose $\gamma\in\sn$ is a compact spherically convex body such that set $\balpha_K(\gamma)\cup\balpha_K(\gamma)$ is contained in open hemisphere $\Omega$. Then, 
%	
%		\begin{equation}\label{strictly 1}
%		\balpha_K(\gamma)\triangle\balpha_L(\gamma) \setminus \big (\balpha_K(\partial\gamma)\cup\balpha_L(\partial\gamma)\big )=\big (\balpha_K(\interior \gamma)\setminus\balpha_L(\gamma)\big )\cup \big (\balpha_L(\interior \gamma)\setminus \balpha_L(\gamma)	\big )\end{equation}
	
	\begin{equation}\label{strictly 2}
		  \balpha_K(\gamma)\triangle\balpha_L(\gamma)\subset\bigcup_{0\leq t\leq1}\balpha_{K\hat +_tL}(\partial\gamma).
	\end{equation}
	
	\begin{equation}\label{strictly 3}
	\balpha_K(\gamma)\triangle\balpha_L(\gamma) \setminus \big (\balpha_K(\partial\gamma)\cup\balpha_L(\partial\gamma)\big )\subset  \bigcup_{0< t<1}\balpha_{K\hat +_tL}(\partial\gamma)
\end{equation}
\end{prop}
\begin{proof}
%	If $K,L$ have strictly convex duals, they satisfy Proposition \ref{boundary for C^1} which immediately implies \eqref{strictly 1}.
	By taking \eqref{strictly 2} and subtracting $\big (\balpha_K(\partial\gamma)\cup\balpha_L(\partial\gamma)\big )$ from both sides we obtain \eqref{strictly 3}. We have to show \eqref{strictly 2}. 
	
	Since $\balpha_K$ and $\balpha_L$ are homeomorphisms we obtain from Proposition \ref{ess properties} that \begin{equation}
		\balpha_{K\hat+_tL}(u)=P((1-t)\balpha_K(u)+t\balpha_L(u)).
	\end{equation}
	From Lemma 3.5 in \cite{Semenov GIP} there exist $\beta>0$ such that for any $u\in\sn$, \begin{equation}
		\balpha_K(u),\balpha_L(u)\subset u_{\frac \pi 2 - \beta}. 
	\end{equation} 
	Thus, given any $u\in\sn$ there exist a $\delta>0$ such that $\balpha_K(u_{\delta}),\balpha_K(u_{\delta})\subset u_{\frac \pi 2 - \beta+\delta}$ with $\frac \pi 2 - \beta+\delta<\frac \pi 2$. Let $\Omega$ be a an open hemisphere containing both $\balpha_K(u_{\delta}),\balpha_K(u_{\delta})$. Therefore, for any $t\in[0,1]$ and $v\in u_\delta$ we have that $(1-t)\balpha_K(v)+t\balpha_L(v)$ has a lower bound on a distance to the center which makes $P((1-t)\balpha_K(v)+t\balpha_L(v))$ a continuous map in variables $(t,u)$ for the region $[0,1]\times u_\delta$.
	Therefore $\balpha_{K\hat+_{\cdot}L}(\cdot)$ is continuous in $(t,u)$ on $[0,1]\times \sn$. Thus, $\balpha_{K\hat+_{\cdot}L}(\cdot)$ defines a homotopy of maps $\balpha_K(\cdot),\balpha_L(\cdot):\sn \rightarrow \sn.$ Note that since $\gamma$ is a compact spherically convex body it is homeomorphic to $D^{n-1}$.
	Now, suppose \eqref{strictly 2} doesn't hold. Then, without loss of generality there exist an $x\in \sn$ such that \begin{equation} 
	x \in \balpha_K(\gamma)\setminus(\balpha_L(\gamma)\cup \bigcup_{0\leq t\leq1}\balpha_{K\hat +_tL}(\partial\gamma)) 
	\end{equation}
	Therefore, from Proposition \ref{boundary for C^1}, $x\in \balpha_K(\interior \gamma)$. Let $p$ be a continuous map from $\Omega\setminus \{x\}\rightarrow \partial \Omega$. Then $p\circ \balpha_{K\hat+_{\cdot}L}(\cdot):[0,1]\times \partial \gamma\rightarrow \partial \Omega$ is a homotopy of maps $p\circ \balpha_K$ and $p\circ\balpha_L$ from $\partial \gamma \cong \mathbb S^{n-2}$ into $\partial \Omega \cong \mathbb S^{n-2}$. Since $x\in \interior{\balpha_K(\gamma)}$, the set $p\circ\balpha_K(\partial\gamma)=\partial\Omega$ and, thus, $\deg(p\circ \balpha_K)\neq 0$. Now, $p\circ \balpha_L:\partial\gamma\rightarrow \partial \Omega$ is a continuous map from $\mathbb S^{n-2}$ to $\mathbb S^{n-2}$ which extends to a continuous map from $p\circ \balpha_L: \gamma \cong \dn \rightarrow \partial \Omega \cong \mathbb S^{n-1}$ since $x\notin \balpha_L(\gamma)$, which forces $\deg (p\circ \balpha_L)=0$. This contradicts Hopf Theorem as homotopic maps $p\circ \balpha_K$ and $p\circ \balpha_L$ have different degrees. This completes the proof of \eqref{strictly 2}.

\end{proof}

Now we are going to prove the same statement without $C^1$ strict convexity assumption using Lipschitz properties of  $\balpha_{K\hat{+}_tL}$.
\begin{lem}\label{sub}
Let $K,L\in \kno$. Suppose $\gamma\in\sn$ is a compact spherically convex body, then the following holds:
\begin{equation}
	\balpha_K(\gamma)\triangle\balpha_L(\gamma) \setminus \big (\balpha_K(\partial\gamma)\cup\balpha_L(\partial\gamma)\big )\subset  \bigcup_{0< t<1}\balpha_{K\hat +_tL}(\partial\gamma)
\end{equation}
%For $\gamma$ in Lemma \ref{gamma}, \begin{equation}
% \balpha_K(\gamma)\setminus\big(\balpha_L(\gamma)\cup\balpha_K(\partial\gamma)\big)\subset\bigcup_{0<t<1}\balpha_{K\hat +_tL}(\partial\gamma)
%\end{equation}
% Similarly, interchanging $K$ and $L$ on the left side of the statement. 
\end{lem}
\begin{proof}
Notice that by symmetry it is enough to show that
\begin{equation}
 \balpha_K(\gamma)\setminus\big(\balpha_L(\gamma)\cup\balpha_K(\partial\gamma)\big)\subset\bigcup_{0<t<1}\balpha_{K\hat +_tL}(\partial\gamma).
\end{equation}
Suppose the set on the left side of this equation is nonempty. 

First, assume that there exist $x_\eps \subset\balpha_K(\gamma)\setminus\big(\balpha_L(\gamma)\bigcup\balpha_K(\partial\gamma)\big)$ for some $\eps>0$ which doesn't intersect $\bigcup_{0<t<1}\balpha_{K\hat +_tL}(\partial\gamma)$. Any two sets $A$ and $B$ such that $x_\eps \subset A$ and $x_\eps \subset B^{c}$ are at least $\eps $ Hausdorff distance from each other. Note that, $\balpha_{K\hat{+}_0L}(\gamma)=\balpha_K(\gamma)$ contains $x_\eps $ and  $\balpha_{K\hat{+}_1L}(\gamma)=\balpha_L(\gamma)$ doesn't. Note that $\partial\balpha_{K\hat +_tL}(\gamma)\subset\balpha_{K\hat +_tL}(\partial\gamma)$ from Proposition \ref{boundary subset}. Since for all $t$, $\balpha_{K\hat +_tL}(\partial\gamma)$ doesn't intersect a ball, we have that for all $t$, $\balpha_{K\hat +_tL}(\gamma)$ either contains the ball or not. From Proposition \ref{continuity of sum} we know that $\balpha_{K\hat{+}_tL}(\gamma)$ is  continuous but above argument shows that $\balpha_{K\hat{+}_tL}(\gamma)$ has at least one $\eps $-jump discontinuity. Contradiction.

Note that from Proposition \ref{boundary subset}, $\partial\balpha_K(\gamma)\subset \balpha_K(\partial\gamma)$. Therefore,
\begin{equation}
	\balpha_K(\gamma)\setminus\big(\balpha_L(\gamma)\cup\balpha_K(\partial\gamma)\big)=\interior{\balpha_K(\gamma)}\setminus\big(\balpha_L(\gamma)\cup\balpha_K(\partial\gamma)\big)
\end{equation}
Since, $\gamma$ and $\partial\gamma$ are closed sets, from previous equation we obtain that $\balpha_K(\gamma)\setminus\big(\balpha_L(\gamma)\bigcup\balpha_K(\partial\gamma)\big)$ is an open set. Now, let $x\in\balpha_K(\gamma)\setminus\big(\balpha_L(\gamma)\bigcup\balpha_K(\partial\gamma)\big)$. We can define $f(t):=d(x,\balpha_{K\hat +_tL}(\partial\gamma))$ on compact set $[0,1]$. Proposition \ref{continuity of sum} guarantees us that $f$ is continuous. Notice that $f(0)>1$ as $\balpha_K(\gamma)\setminus\big(\balpha_L(\gamma)\bigcup\balpha_K(\partial\gamma)\big)$ is an open set and $\balpha_{K\hat +_0L}(\partial\gamma)=\balpha_K(\partial\gamma)$. Similarly, $f(1)>0$ as $\balpha_{K\hat +_1L}(\partial\gamma)=\balpha_L(\partial\gamma)\subset \balpha_L(\gamma)$. Suppose $c$ is minimum of $f$ on $[0,1]$. If $c>0$ we have a contradiction from the previous argument by setting $r=\frac c 2$. If $c=0$ this means that for some $t'$ we have $d(x,\balpha_{K\hat +_{t'}L}(\partial\gamma))=0$, and thus $x\in\balpha_{K\hat +_{t'}L}(\partial\gamma)$ as $\balpha_{K\hat +_{t'}L}(\partial\gamma)$ is a closed set. 
%TODDO really important to mention that closed sets map into closed sets.

\end{proof}

The proof of the next Lemma is very similar to a proof of Lemma 3.8 in \cite{GIP} and can be considered its natural extension.  

\begin{lem}\label{eta_o}
	Suppose $K,L\in\kno$ and we are given $u\in\sn$. %Let $K'$ be rescaling of $K$ such that there exist a point $r_{K'}(u)=r_L(u)$. 
	Define, 
	\begin{equation}\label{eta}
		\eta(u):=\bigcup_{0<t<1}\balpha_{K\hat +_tL}(u)\setminus\big(\balpha_K(u)\cup\balpha_L(u)\big).
	\end{equation}
%	Then there exist an open set $\eta_o(u)\subset \sn$ such that,
%	\begin{enumerate}
%		\item $\eta\subset \eta_o(u)\subset \balpha_K(\rho_{K'}>\rho_L)$
%		\item $\balpha_L(\rho_{K'}>\rho_L)\subset\balpha_K(\rho_{K'}>\rho_L)\setminus \eta_o(u)$
%	\end{enumerate}
If $\eta(u)\neq \emps$ then there exist an open set $\eta_o(u)$ containing $\eta(u)$ such that if $\lambda(\eta_o(u))>0$ then $\lambda(K,\cdot)\neq\lambda(L,\cdot)$.
\end{lem}
%\begin{rmk}
%		In particular, if $\lambda(\eta_o(u))>0$ then $\lambda(K,\cdot)$ is not equal to $\lambda(L,\cdot)$.
%\end{rmk}
\begin{proof}
	Let $K'$ be rescaling of $K$ such that $r_{K'}(u)=r_L(u)$. First notice that $\balpha_K=\balpha_{K'}$. Moreover, since for any $v_1,v_2\in\sn$ contained in same open hemisphere and any constant $c>0$: \begin{equation}
		\bigcup_{0<t<1}P((1-t)v_1+tv_2)=\bigcup_{0<t<1}P((1-t)cv_1+tv_2),
	\end{equation}
	using Proposition \ref{ess properties}, we obtain:
	  \begin{equation}
		\bigcup_{0<t<1}\balpha_{K\hat +_tL}(u)=\bigcup_{0<t<1}\balpha_{K'\hat +_tL}(u).
	\end{equation}
	Thus, we can substitute $K$ with $K'$ in \eqref{eta}. Moreover, $\lambda(K,\cdot)=\lambda(K',\cdot)$. Therefore, to prove this proposition we can safely assume that $K=K'$ for simplicity. In other words, we can assume $r_{K}(u)=r_L(u)$.
	
	Define $\omega',\omega, \omega_0, \eta_o(u)$ to be the following sets:
	\begin{equation}\label{defi in GIP}
		\begin{split}
			\omega'&:=\{\rho_{K}>\rho_L\}:=\{v\in\sn\mid\rho_K(v)>\rho_L(v)\} \\
			\omega&:=\{\rho_{K}<\rho_L\}:=\{v\in\sn\mid\rho_K(v)<\rho_L(v)\} \\ 
			\omega_o&:=\{\rho_{K}=\rho_L\}:=\{v\in\sn\mid\rho_K(v)=\rho_L(v)\} \\
			\eta_o(u)&:=(\sn\setminus\balpha_K(\omega\cup\omega_0))\cap(\sn\setminus\balpha_L(\omega'\cup\omega_0))
		\end{split}
	\end{equation}
	Clearly $\omega',\omega, \omega_0$ are Borel sets. Let's show that $\eta_0(u)$ is an open set. Since the radial function is continuous, sets $\omega\cup\omega_0$ and $\omega'\cup\omega_0$ are closed. Since radial Gauss image map maps closed sets to closed sets, $\balpha_K(\omega\cup\omega_0)$ and $\balpha_L(\omega'\cup\omega_0)$ are closed sets. Now, invoking the definition of $\eta_o(u)$ we immediately  get that $\eta_0(u)$ is an open set.
	
	Now, we are going to show three statements which together with $\lambda(\eta_o(u))>0$ guarantee  $\lambda(K,\omega')\neq\lambda(L,\omega')$, which, in turn, implies desired $\lambda(K,\cdot)\neq\lambda(L,\cdot)$. The statements are the following:
	\begin{equation}\label{stat 1}
		\balpha_L(\omega')\subset \balpha_K(\omega')
	\end{equation}
	\begin{equation}\label{stat 2}
		\eta_o(u)\subset \balpha_K(\omega')
	\end{equation}
	\begin{equation}\label{stat 3}
		\balpha_L(\omega')\subset\balpha_K(\omega')\setminus \eta_o(u)
	\end{equation}

Let's start with \eqref{stat 1}. Pick any $v\in \balpha_L(\omega')$. Let $w\in\omega'$ be such that $v\in \balpha_L(w)$. Since $\rho_K(w)>\rho_L(w)$ we obtain: \begin{equation}
		h_K(v)\geq vr_K(w)>vr_L(w)=h_L(v).
	\end{equation} 
	Thus, $H_K(v)$ is further away from center than $H_L(v)$. Let $x\in\balpha_K^*(v)$. Since $r_K(x)\in H_K(v)$, we obtain that $x\in\omega'$, and, thus, $v\in\balpha_K(\omega')$, which shows \eqref{stat 1}. Equation \eqref{stat 2} immediately follows from definitions \eqref{defi in GIP} since $\{\omega',\omega, \omega_0\}$ is a partition of $\sn$. Notice that \eqref{stat 2} also shows that $\eta_o(u)\cap \balpha_L(\omega')=\emps$. This combined with \eqref{stat 1} implies \eqref{stat 3} which completes the proof of all three equations. Now, if $\lambda(\eta_o(u))>0$ then \eqref{stat 1}, \eqref{stat 2} and \eqref{stat 3} combined imply that $\lambda(K,\omega')\neq\lambda(L,\omega')$, and, hence, $\lambda(K,\cdot)\neq \lambda(L,\cdot)$.

%	Since we have already showed that $\balpha_L(\rho_K>\rho_L)\subset \balpha_K(\rho_K>\rho_L)$ as long as $\eta_0\cap\balpha_L(\rho_K>\rho_L)=\emps$ we will provide the second part of the statement. This clearly follows from the way we defined $\eta_0$ in \eqref{def of eta}, and, thus, $\balpha_L(\rho_{K}>\rho_L)\subset\balpha_K(\rho_{K}>\rho_L)\setminus \eta_o(u)$. From the definition \eqref{def of eta} we again clearly obtain that $\eta_o(u)\subset \balpha_K(\rho_K>\rho_L)$. 
%	

We are left to show that $\eta(u)\subset \eta_o(u)$. Pick any $v\in \eta(u)$. Using definition of $\eta(u)$ and Proposition \ref{ess properties} we obtain that $v=P((1-t)v_1+tv_2)$ for some $t\in(0,1),v_1\in F(K^*,u),v_2\in F(L^*,u)$ and $v\notin(\balpha_K(u)\cup\balpha_L(u))$. Let $x_0:=r_K(u)=r_L(u)$ and $P=\{x\in\rn\mid xv=x_0v\}$. 

We want to show that if for some $w\in\sn$, $r_K(w)\in H_K(v)$, then $w\in \omega'$. Let $x'$ be orthogonal projection of $r_K(w)$ onto $P$. Let
\begin{equation}
	r_K(w)=x'+cv
\end{equation}
for some constant $c$. Since $v\notin\balpha_K(u)$ we have that \begin{equation}
\begin{split}
	x'v+c&=(x'+cv)v=r_K(w)v=h_K(v)>r_K(u)v=x_0v \Leftrightarrow \\
	c&>x'v-x_0v=0
\end{split}
\end{equation}  
where the last equality holds since $x'\in P$. Now, since $v_1\in F(K^*,u)$ and $x_0=r_K(u)$ we write
\begin{equation}
	{x_0v_1} =h_K(v_1)\geq r_K(w)v_1=x'v_1+cvv_1=x'v_1+c \frac {(1-t)v_1+tv_2} {\lVert (1-t)v_1+tv_2\rVert }v_1>x'v_1
\end{equation}
where the last step we used $c>0$ and the fact  that$v,v_1,v_2$ are contained in some $u_{\frac \pi 2 - \eps}$ for some epsilon from Lemma 3.8 in \cite{Semenov GIP}, which forces $vv_1$ and $vv_2$ to be positive. 

We obtain
\begin{equation}\label{one part}
	x_0v_1 > x'v_1.
\end{equation}
Now, 
\begin{equation}\label{comput}
\begin{split}
	x'v&=x_0v \Leftrightarrow \\
	x'\frac {(1-t)v_1+tv_2} {\lVert (1-t)v_1+tv_2\rVert }&=x_0\frac {(1-t)v_1+tv_2} {\lVert (1-t)v_1+tv_2\rVert }
\end{split}
\end{equation}
Using \eqref{one part} and recalling that $0<t<1$ we obtain that \eqref{comput} implies
\begin{equation}
	x_0v_2<x'v_2.
\end{equation}
Therefore, using previous proposition and the fact that $v_2\in F(L^*,u)$ and $x_0=r_L(u)$, we obtain
\begin{equation}
	r_K(w)v_2=x'v_2+cvv_2>x_0v_2=h_L(v_2).
\end{equation}
Thus, $r_K(w)\notin L$, and, hence, $r_K(w)>r_L(w)$ which implies that $w\in \omega'$ which is what we wanted to show.

Now, since in the last paragraph we showed that there doesn't exist $w\in \omega\cup\omega_0$ such that $r_K(w)\in H_K(v)$, we have that $v\notin\balpha_K(\omega\cup\omega_0)$. In the same way, we obtain that $v\notin\balpha_L(\omega'\cup\omega_0)$ and, thus, $v\in \eta_o(u)$. We established that $\eta(u)\subset \eta_o(u)$ and this completes the proof. 
\end{proof}

We are ready to show Theorem \ref{main}.

 \begin{thm}
 Let $K,L\in \kno$. Suppose $\lambda(K,\cdot)=\lambda(L,\cdot)$ are finite Borel measures for a spherical Lebesgue submeasure $\lambda$. Then $\forall \omega\subset S^{n-1}$ Borel sets $\lambda(\balpha_K(\omega)\triangle\balpha_L(\omega))=0$, where $\triangle$ denotes the symmetric difference between two sets.
 \end{thm}
 \begin{proof}
 	Suppose there exist $\omega$ Borel such that $\lambda(\balpha_K(\omega)\triangle\balpha_L(\omega))>0$. Then, using Lemma \ref{gamma} there exist a compact spherically convex body $\gamma$ such that \begin{equation}\label{final steps}
 		\lambda \Big (\balpha_K(\gamma)\triangle\balpha_L(\gamma) \setminus \big (\balpha_K(\partial\gamma)\cup\balpha_L(\partial\gamma)\big )\Big )>0
 	\end{equation}
 	Recall that we proved in Lemma \ref{sub} (or Lemma \ref{same for strictly} instead if we assume $K,L$ are $C^1$ strictly convex bodies) that for $\gamma$ compact spherically convex body: 
 	\begin{equation}
	\balpha_K(\gamma)\triangle\balpha_L(\gamma) \setminus \big (\balpha_K(\partial\gamma)\cup\balpha_L(\partial\gamma)\big )\subset  \bigcup_{0< t<1}\balpha_{K\hat +_tL}(\partial\gamma)
\end{equation}
This implies 
\begin{equation}\label{4.50}
\begin{split}
	&\balpha_K(\gamma)\triangle\balpha_L(\gamma) \setminus \big (\balpha_K(\partial\gamma)\cup\balpha_L(\partial\gamma)\big )\subset  \bigcup_{0< t<1}\balpha_{K\hat +_tL}(\partial\gamma)\setminus\big (\balpha_K(\partial\gamma)\cup\balpha_L(\partial\gamma)\big ) \subset \\
	& \bigcup_{u\in \partial \gamma}  \bigcup_{0< t<1}\balpha_{K\hat +_tL}(u)\setminus\big (\balpha_K(\partial\gamma)\cup\balpha_L(\partial\gamma)\big ) \subset \\
	 &\bigcup_{u\in \partial \gamma} \Big ( \bigcup_{0< t<1}\balpha_{K\hat +_tL}(u) \setminus \big (\balpha_K(u)\cup\balpha_L(u) \big ) \Big ) = \\
	 & \bigcup_{u\in \partial \gamma} \eta(u) \subset \\
	 & \bigcup_{\substack{u\in \partial \gamma \\ \eta(u)\neq \emps} } \eta_o(u)
\end{split}
\end{equation}
where in the last line we used notation from Lemma \ref{eta_o}. As $\sn$ is a second countable space every subspace of it is also second countable and, in particular, Lindel\"{o}f. From \eqref{4.50} the set $\{\eta_o(u)\mid u\in \partial \gamma \text{ and } \eta(u)\neq \emps\}$ provided an open cover for the set in \eqref{final steps}, and, hence, there is a countable (not necessarily finite) subcover $\{\eta_0(u_i)\mid u_i\in\partial \gamma\ \text{ and } i\in\mathbb{N}\}$ covering the set from \eqref{final steps}. Therefore, from equation \eqref{final steps} we obtain
\begin{equation}
	\lambda(\bigcup_{i\in \mathbb{N}}\eta_o(u_i))>0.
\end{equation}
Therefore, since $\lambda$ is subadditive at least one of the sets $\eta_o(u_i)$ has a positive measure. Invoking Lemma \ref{eta_o} for this set we obtain that $\lambda(K,\cdot)\neq \lambda(L,\cdot)$.
 \end{proof}

\section{Applications}\label{Section Appl}

We now turn to applications of the main theorem and machinery established. We will show that the measure-theoretic relation between $\balpha_K$ and $\balpha_L$ implies geometric relation between bodies $K^*$ and $L^*$. If we consider submeasure $\lambda$ supported on a very small set it is clear that bodies $K$ and $L$ can be very different at those radial directions $\omega$ for which $\lambda(K,\omega)=\lambda(L,\omega)=0$. Yet, we are going to show that the bodies are scalings of each other when we restrict our attention to support of submeasure $\lambda$.

Recall that the support of submeasure $\lambda$ is defined as
\begin{equation}\label{Definition of supporrt}
	\spt\lambda=\{v\in\sn\mid \text{for every open neighborhood }N_v \text{ of v, } \lambda(N_v)>0\}
\end{equation}
The support of submeasure is known to be a closed set. Consider the following Lemma:

\begin{lem}\label{gradient}
	Given $K,L\in\kno$ and $u\in\sn$. Suppose $\balpha_K^*(u)=\balpha_L^*(u)=v$. Then \begin{equation}\label{gradient vanish}
		\nabla\frac{\rho_{K^*}}{\rho_{L^*}}(u)=0.
	\end{equation}
\end{lem}
\begin{proof}
	We omit the proof of this result and it naturally follows from differentiation of radial functions of bodies $K$ and $L$ and is not needed for any result of this work.
\end{proof}
The Lemma \ref{gradient} was used in \cite{GIP} to show that if $\lambda$ is absolutely continuous, strictly positive on open sets and $\lambda(K,\cdot)=\lambda(L,\cdot)$ then $K$ is scaling of $L$. Using similar ideas one possible relaxation would be to establish \eqref{gradient vanish}
on $\spt \lambda$ from Theorem \ref{main}. Then, using  equation \eqref{gradient vanish} prove that $K^*$ and $L^*$ are dilates on any connected component $D$ of $\spt \lambda $. However, this approach doesn't work for the full generality of the problem. 

First of all, if submeasure $\lambda$ is supported on a small set there might be  no $u\in \spt \lambda$ such that $\balpha_K(u)$ and $\balpha_L(u)$ are vectors. But even then, one can construct an example of $K,L\in\kno$ such that $\lambda(K,\cdot)=\lambda(L,\cdot)$ and for all $u\in\spt \lambda$, $\balpha_K^*(u)\neq\balpha_L^*(u)$ which would imply that \eqref{gradient vanish} doesn't hold anywhere on $\spt \lambda$.
   
More interestingly, let us address a very curious example constructed by Hassler Whitney in 1935, see \cite{Whitney}. He constructed a $C^1$ function $f$ on the plane such that $\nabla f=0$ on an arc $D\subset \mathbb {R}^2$ and, yet, $f$ is not constant on $D$. Clearly, something like this is not possible on a real line or if $D$ is some $C^1$ curve. The arc $D$ in Whitney's example is a fractal curve that
has an infinite length between any two points. In particular, knowledge of \eqref{gradient vanish} at any point of path-connected component $D\subset \spt \lambda$ might not be sufficient to insure that $K^*$ is a dilate of $L^*$ on $D$.

Guided by these ideas we are going to show that if $\lambda(K,\cdot)=\lambda(L,\cdot)$ then $K^*$ is a dilate of $L^*$ on each rectifiable path-connected component. Let us start with the definition of rectifiability.

\begin{defi}
	 Let $\Pi$ be a set of all partitions of $[0,1]$ such that if $\pi \in \Pi$ then \begin{equation}
	 	\pi=(t_0,t_1,\ldots ,t_k)
	 \end{equation}
	 with $t_0=0$ and $t_k=1$. Given $\pi \in \Pi$ and a map $\gamma:[0,1]\rightarrow \sn$ define 
	 
	 \begin{equation}
	 	l(\pi,\gamma)= \sum _{i=1}^{k}d\big (\gamma(t_i),\gamma(t_{i+1})\big ).
	 \end{equation}
	 
	 A continuous map $\gamma:[0,1]\rightarrow \sn$ is called rectifiable curve if
	\begin{equation}
		L(\gamma)=\sup_{\pi \in \Pi}l(\pi,\gamma)<+\infty 
	\end{equation} 
\end{defi}

Now we define a stronger version of a path-connected set to exclude Whitney's example from our discussion. 

\begin{defi}\label{rectifiable path connected}
	The set $D\subset \sn$ is called a rectifiable  path connected if given any $x,y\in D$ there exists a rectifiable path $\gamma:[0,1]\rightarrow D$ such that $\gamma(0)=x$ and $\gamma(0)=y$.  
\end{defi}

With this definition in mind, our main result of this section is Theorem \ref{main 2}:

\begin{thm}\label{main 21}
	
	Let $K,L\in\kno$  such that $\lambda(K,\cdot)=\lambda(L,\cdot)$ are finite Borel measures for a spherical Lebesgue sumbeasure $\lambda$. Then on each rectifiable  path connected component $D\subset \spt \lambda$, $K^*$ is a dilate of $L^*$. Alternatively, for each $v_1,v_2\in D$ we have \begin{equation}
		\frac{h_K(v_1)}{h_L(v_1)}=\frac{h_K(v_2)}{h_L(v_2)}
	\end{equation}
\end{thm}
\begin{rmk}
	In particular, one can think about this in terms of tangential bodies.  
\end{rmk}

Note that, outside of fractal examples, the rectifiable condition is not such a strong assumption. For example, any open connected subset of $\spt \lambda$ is rectifiable path-connected. %TODO add more ? Think ? What measures have rectifiable path-connected supports? Add this to beginning. I definitely need to prove for lambda AC, this statement? Hausdorf 1 dimensional
 In particular, this shows that the uniqueness of the Gauss Image Problem is a local property compared to the existence result for which you need global information. This result immediately implies the uniqueness results established in \cite{GIP} and \cite{Aleks}. The only thing which can go wrong is if measure $\lambda$ is supported on some sort of fractal similar to arc from example in \cite{Whitney}.  
 
 We start with some preliminary lemmas.
 
 \begin{lem}\label{nonempty}
 	Let $K,L\in \kno$. Suppose $\lambda(K,\cdot)=\lambda(L,\cdot)$ are finite Borel measures for a spherical Lebesgue submeasure $\lambda$. Then for any $u\in\spt \lambda$ \begin{equation}
 		\balpha_{K^*}(u)\cap \balpha_{L^*}(u)\neq \emps.
 	\end{equation}
 \end{lem}
 \begin{proof} 
 	Suppose not. Since $\balpha_{K^*}(u), \balpha_{L^*}(u)$ are compact sets which don't intersect there exist and $\eps >0$ such that \begin{equation}\label{5.5 equation 0}
 		\eps =2\inf_{\substack{x\in \balpha_{K^*}(u) \\ y\in \balpha_{L^*}(u)}} d(x,y).
 	\end{equation}
From continuity of radial Gauss Image map, Proposition \ref{continuity}, we obtain that there exist $\delta>0$ such that \begin{equation}\label{5.5 equation 1}\begin{split}
 		\balpha_{K^*}(u_\delta)\subset \balpha_{K^*}(u)_\eps \\  
 		\balpha_{L^*}(u_\delta)\subset \balpha_{L^*}(u)_\eps	\end{split}
 	\end{equation}
 	We rewrite \eqref{5.5 equation 0} in the form
 	\begin{equation}\label{5.5 don't intersetct}
 		\balpha_{K^*}(u)_\eps \cap \balpha_{L^*}(u)_\eps = \emps.
 	\end{equation}
 	
 Let $\omega=\balpha_{K^*}(u)_\eps$. Then \eqref{5.5 equation 1} implies 
 \begin{equation}\label{5.5 one more}
 	u_\delta \subset \balpha_{K}(\balpha_{K^*}(u_\delta))\subset \balpha_{K}(\balpha_{K^*}(u)_\eps)= \balpha_{K}(\omega).
 \end{equation}
Suppose $u_\delta \cap \balpha_L(\omega)$ is not empty. Let $v\in u_\delta \cap \balpha_L(\omega)$. Then using \eqref{5.5 equation 1} we obtain
\begin{equation}
	\balpha_{L^*}(v)\subset \balpha_{L^*}(u_\delta)\subset \balpha_{L^*}(u)_\eps 
\end{equation}
Thus, from \eqref{5.5 don't intersetct}, $\balpha_{L^*}(v)\cap \omega = \emps$ which contradicts that $v\in \balpha_L(\omega)$. This establishes that
\begin{equation}\label{5.5 another one}
	u_\delta \cap \balpha_L(\omega)=\emps. 
\end{equation}
 
 Combining, \eqref{5.5 one more} and \eqref{5.5 another one} we obtain \begin{equation}
 	u_\delta\subset \balpha_K(\omega) \triangle\balpha_L(\omega).
 \end{equation}
 Since $u_\delta\cap \spt \lambda \neq \emps$ and $u_\delta$ is open, this implies \begin{equation}
 	\lambda(\balpha_K(\omega)\triangle\balpha_L(\omega))>0
 \end{equation}
 which contradicts Theorem \ref{main}.
 \end{proof}
 
 \begin{lem}\label{continuity of K and L}
 	Let $K,L\in \kno$. Suppose $\lambda(K,\cdot)=\lambda(L,\cdot)$ are finite Borel measures for a spherical Lebesgue submeasure $\lambda$. Let $u\in\spt \lambda$. Then for any $\eps>0$ there exist $\delta >0$ such that \begin{equation}
 		\balpha_{K^*}(u_\delta)\cap \balpha_{L^*}(u_\delta)\subset (\balpha_{K^*}(u)\cap\balpha_{L^*}(u))_\eps
 	\end{equation}
 \end{lem}
 \begin{proof}
 	By continuity, Proposition \ref{continuity}, given any $\delta>0$, there exist $\delta_1>0$ such that 
 	\begin{equation}
 		\balpha_{K^*}(u_{\delta_1})\cap \balpha_{L^*}(u_{\delta_1})\subset \balpha_{K^*}(u)_\delta\cap \balpha_{L^*}(u)_\delta.
 	\end{equation} 
 	Thus, it is sufficient to prove that given $\eps >0$ there exist $\delta>0$ such that \begin{equation}\label{to show 5.6}
 		\balpha_{K^*}(u)_\delta\cap \balpha_{L^*}(u)_\delta\subset   (\balpha_{K^*}(u)\cap\balpha_{L^*}(u))_\eps
 	\end{equation}
 	Let $\omega=\balpha_{K^*}(u)$ and $\gamma=\balpha_{L^*}(u)$. Recall that $\omega$ and $\gamma$ are compact spherically convex sets. Suppose \eqref{to show 5.6} doesn't hold for some $\eps $, then there exist a sequence $x_n\in \omega_\delta \cap \gamma_\delta \setminus (\omega\cap\gamma)_\eps$ where $\delta = \frac 1 n$. By compactness there exists an $x\in\sn$ such that $x_n\rightarrow x$ along some subsequence. Since $x_n\in \omega_\delta$, and $\omega$ is a closed set, $x\in \omega$. Similarly, $x\in\gamma$, and, thus,  $x\in \omega\cap \gamma$. Therefore, for large enough $n$, $x_n\in  (\omega\cap\gamma)_\eps$ which is a contradiction.  
 \end{proof}
 
 These two Lemmas characterize the continuity of normal cones on $\spt \lambda$. 
 
 \begin{defi}
 By $\balpha_{K^*,L^*}(\omega)$ we denote the simultaneous radial Gauss Image map of bodies $K^*$ and $L^*$ as: \begin{equation}
 	\balpha_{K^*,L^*}(\omega)=\balpha_{K^*}(\omega)\cap\balpha_{L^*}(\omega).
 \end{equation}
 \end{defi}
 
With this definition we obtain Theorem \ref{new additional main}:
\begin{prop}\label{common continuity}
	Let $K,L\in \kno$. Suppose $\lambda(K,\cdot)=\lambda(L,\cdot)$ are finite Borel measures for a spherical Lebesgue submeasure $\lambda$. Then, $\balpha_{K^*,L^*}$ defined on $\spt \lambda$ is a continuous map. That is, for any $\eps>0$ there exist $\delta >0$ such that for any $u\in\spt \lambda$ \begin{equation}
 		\balpha_{K^*,L^*}(u_\delta)\subset \balpha_{K^*,L^*}(u)_\eps.
 	\end{equation}
\end{prop}
\begin{proof}
	This follows from Lemma \ref{nonempty} and Lemma \ref{continuity of K and L}. 
\end{proof}

The next lemma is an analog of Lemma \ref{gradient} for general radial Gauss Image maps. This roughly means that if $u\in\spt \lambda$ if we restrict our attention to $\spt \lambda$, then radial functions roughly behave like $\nabla \frac{\rho_{K^*}}{\rho_{L^*}}(u)=0$. Since $\spt\lambda$ might not have a differentiable structure we use the equivalent formulation:

%Todo think about Lebesgue differentiation theorem with relation to this. Lebesgue differentiation will give you that for AC function supp of measure is support of function.

\begin{lem}\label{derivative is zero}
	Let $K,L\in \kno$. Suppose $\lambda(K,\cdot)=\lambda(L,\cdot)$ are finite Borel measures for a spherical Lebesgue submeasure $\lambda$. Then, for each $u\in\spt\lambda$ and any $\eps >0$ there exist $\delta>0$ such that for all $u'\in u_\delta\cap \spt \lambda$: \begin{equation}\label{long equation}
		|\frac {\rho_{K^*}}{\rho_{L^*}}(u')-\frac {\rho_{K^*}}{\rho_{L^*}}(u)| \leq \eps |u'-u|.
	\end{equation}
\end{lem}
\begin{proof}
Before finding $\delta$ let us establish some a priori estimated as $\delta''$ which we will use later in the proof. If $u$ is an isolated point of $\spt \lambda$ there is nothing to prove. Pick $\delta'>0$ such that $\balpha_{K^*,L^*}(u)_{\delta'}$ is contained in $u_{\frac \pi 2 - \eps'}$ for some $\eps'>0$. Using Proposition \ref{common continuity} consider neighborhood $u_{\delta''}$ such that if $u'\in u_{\delta''} \cap \spt \lambda$, then $\balpha_{K^*,L^*}(u')\subset \balpha_{K^*,L^*}(u)_{\delta'}\subset u_{\frac \pi 2 - \eps'}$. This guarantees that  for any $v\in \balpha_{K^*,L^*}(u)$ and $v'\in \balpha_{K^*,L^*}(u')$ we have ${v,v',u,u'}\subset u_{\frac \pi 2 - \eps'}$ and hence a uniform positive lower bound on their inner products. Moreover, we can pick $\delta''$ so that: \begin{equation}\label{stupid bound}
  u' \in u_{\delta''} \cap \spt \lambda \Rightarrow \rVert u'\lambda+(1-\lambda)u \rVert>\frac 1 2 
  \end{equation}

	With $\delta''$ in hand, let us start the proof. Pick any $u'\in u_{\delta''} \cap \spt \lambda$ and any $v\in \balpha_{K^*,L^*}(u)$ and $v'\in \balpha_{K^*,L^*}(u')$. Then, \begin{equation}\begin{split}
		&\rho_{K^*}(u')u'v\leq h(K^*,v)=\rho_{K^*}(u)uv \\
		&\rho_{L^*}(u')u'v\leq h(L^*,v)=\rho_{L^*}(u)uv \\
		&\rho_{K^*}(u')u'v'=h(K^*,v')\geq \rho_{K^*}(u)uv' \\
		&\rho_{L^*}(u')u'v'=h(L^*,v')\geq \rho_{L^*}(u)uv'.
	\end{split}
	\end{equation}
	Combining these equations we obtain: \begin{equation}\label{two sides}
		\frac {\rho_{K^*}}{r_{L^*}}(u) \frac {uv'}{u'v'}\frac{u'v}{uv} \leq \frac {\rho_{K^*}}{r_{L^*}}(u')\leq \frac {\rho_{K^*}}{\rho_{L^*}}(u) \frac {uv}{u'v}\frac{u'v'}{uv'}.
	\end{equation}
	%The reader familiar with cyclical monotonicity can obviously spot it in the previous equation. TODO this is not true. 
	Define $f: \rn \times \rn \rightarrow \mathbb R$ such that \begin{equation}
		f(x,y)=\frac {uv}{xv}\frac {xy}{uy}
	\end{equation} Note in particular that for any $x,x',y,y'$ in $u_{\frac \pi 2 - \eps'}$:
	\begin{equation}\label{5.21}
	\begin{split}
		f(x,v)=f(x',v) \\
		f(u,y)=f(u,y')
	\end{split}
	\end{equation} Since $f(u,v)=1$ by using \eqref{two sides} and definition of $f$ we obtain,
	\begin{equation}\label{5.20}
	\frac {\rho_{K^*}}{\rho_{L^*}}(u)(\frac 1 {f(u',v')}-1)	\leq \frac {\rho_{K^*}}{\rho_{L^*}}(u')-\frac {\rho_{K^*}}{\rho_{L^*}}(u) \leq \frac {\rho_{K^*}}{\rho_{L^*}}(u)(f(u',v')-1)
	\end{equation}
	In particular, this shows that $f(u',v')-1\geq0$ as $\text{sign}(\frac 1 x - 1)=-\text{sign}(x-1)$.
	
	Note that $f$ is $C^\infty$ function in a neighborhood of $(u,v)$. Holding vector $y$ fixed, the gradient $\nabla_x f$ of $f(x,y)$ at point $(x',y')$ is going to be:
	\begin{equation}\label{gradient of f}
		\nabla_x f (x',y')=\frac {u \cdot v}{u \cdot y'} \frac {(x'\cdot v)y'-(x' \cdot y')v}{(x'\cdot v)^2}
	\end{equation}
	Thus, for some $t\in[0,1]$, using \eqref{5.21} and mean value theorem we obtain
	\begin{equation}\label{some form of f}
		f(u',v')-1=f(u',v')-f(u,v')=\nabla_x f (u+t(u'-u),v')(u'-u)
	\end{equation}
	Denote by $x'=u+t(u'-u)$ we compute the following:
	\begin{equation}\label{computing something}
	\begin{split}
		\lVert{(x'\cdot v)v'- (x' \cdot v')v}\rVert= \\ \lVert(x'\cdot v)(v'-v+v)-(x' \cdot (v'-v+v))v\rVert=\\
		\lVert(x'\cdot v) (v'-v) - (x'\cdot (v'-v))v\rVert \leq  \\
		2\lVert x'\rVert\lVert v\rVert\lVert v'-v\rVert. = \\ 2\lVert x'\rVert\lVert v'-v\rVert.
			\end{split}
	\end{equation}
	Combining \eqref{gradient of f}, \eqref{some form of f} and \eqref{computing something} we obtain that 
	\begin{equation}
		|f(u',v')-1|\leq 2 \frac {u \cdot v}{(u \cdot v')(x'\cdot v)^2} \lVert x'\rVert \lVert v'-v\rVert \lVert u'-u \rVert 
	\end{equation}
	Since we, beforehand, assumed that we work in ${v,v',u,u'}\subset u_{\frac \pi 2 - \eps'}$ and all inner products are larger than some uniform constant and using \eqref{stupid bound} we obtain that there exists a constant $C_{\eps'}$ which only depends on $\eps'$
	\begin{equation}\label{5.28}
		|f(u',v')-1|\leq C_{\eps'}  \lVert v'-v\rVert \lVert u'-u \rVert 
	\end{equation}
	Combining everything together and recalling that $f(u',v')-1\geq 0$ we obtain
	\begin{equation}\label{last to combine}
		\frac {\rho_{K^*}}{\rho_{L^*}}(u')-\frac {\rho_{K^*}}{\rho_{L^*}}(u) \leq \frac {\rho_{K^*}}{\rho_{L^*}}(u) C_{\eps'}  \lVert v'-v\rVert \lVert u'-u \rVert 
	\end{equation}
	By continuity, Proposition \ref{common continuity}, for any $\eps''$ there exist a $\delta<\delta''$ such that for all $u'\in u_{\delta}\cap \spt \lambda$, and for any $v'\in \balpha_{K^*,L^*}(u')$ we have that $v'\in \balpha_{K^*,L^*}(u)_{\eps''}$. Thus we can choose $\eps''$ small enough, such that, for all $u'\in u_{\delta}\cap \spt \lambda$ and all $v'\in \balpha_{K^*,L^*}(u')$ there exist $v\in\balpha_{K^*,L^*}(u)$ with \begin{equation}\label{5.30}
		\frac {\rho_{K^*}}{\rho_{L^*}}(u) C_{\eps'}  \lVert v'-v\rVert < \frac \eps 2 
	\end{equation}
	Combining this with \eqref{last to combine} we obtain that for all $u'\in u_{\delta}\cap \spt \lambda$, \begin{equation}
		\frac {\rho_{K^*}}{\rho_{L^*}}(u')-\frac {\rho_{K^*}}{\rho_{L^*}}(u) \leq \frac \eps 2 \lVert u'-u \rVert
	\end{equation}
Now, notice that 
\begin{equation}
	\frac 1 {f(u',v')}-1= - \frac {f(u',v')-1} {f(u',v')}
\end{equation}
Since $f(x,y)$ is continuous at $(u,v)$ and $f(u,v)=1$ we can insure that our neighborhoods are small enough to claim from \eqref{5.20}, \eqref{5.28}, \eqref{5.30} that
\begin{equation}
	- \eps \lVert u'-u \rVert \leq \frac {\rho_{K^*}}{\rho_{L^*}}(u')-\frac {\rho_{K^*}}{\rho_{L^*}}(u) 
	\end{equation} 
	which completes the proof.
\end{proof}

We are now ready to prove the main result of this section

\begin{proof}[\textbf{Proof of Theorem \ref{main 21}}]
	Since $D$ is rectifiable  path connected component of $\spt \lambda$ given any $u_0,u_1\in D$ there exists a rectifiable path $\gamma(t)$ from some interval into $\spt \lambda$. 
	Since $\gamma$ is rectifiable we can assume that it is arc-length parametrized on interval $[0,L]$, see \cite{Falconer}. Pick $\eps>0$. By Lemma \ref{derivative is zero} for each $u\in \gamma$ there exist $\delta_{u}$ such that for all $u'\in u_{\delta_u}\cap \spt \lambda$: \begin{equation}
		|\frac {\rho_{K^*}}{\rho_{L^*}}(u')-\frac {\rho_{K^*}}{\rho_{L^*}}(u)| \leq \eps |u'-u|.
	\end{equation}
	By continuity for each $t\in[0,L]$ there exist $\delta'_t$ such that for $|t-t'|<\delta_t'$ we have $\gamma(t')\in \gamma(t)_{\delta_{\gamma(t)}}\cap \spt \lambda$ and, hence,
	
	\begin{equation}\label{5.39}
		|\frac {\rho_{K^*}}{\rho_{L^*}}(\gamma(t'))-\frac {\rho_{K^*}}{\rho_{L^*}}(\gamma(t))| \leq \eps |\gamma(t')-\gamma(t)|=\eps|t'-t|.
	\end{equation}
	where the last equality follows from arc-parametrization. Note that $\{I_t=(t-\frac {\delta_t'} 2,t+ \frac {\delta_t'} 2)\mid t\in [0,L]\}$ is an open cover for $[0,L]$. Let $\mathcal{C}=\{I_{t_i}\mid 1\leq i \leq k \text{ with } 0=t_1<t_2<\ldots < t_k=L\}$ be its finite subcover. Note that by construction of $I_t$ for each $1\leq i <k$, 
	\begin{equation}
		t_{i+1}-t_i< \frac {\delta_{t_{i+1}}'} 2 +  \frac {\delta_{t_{i}}'} 2 \leq \max (\delta_{t_i}',\delta_{t_{i+1}}')
	\end{equation}
and, thus, we can apply \eqref{5.39} to conclude that for each $1\leq i< k$,
\begin{equation}
		|\frac {\rho_{K^*}}{\rho_{L^*}}(\gamma(t_{i+1}))-\frac {\rho_{K^*}}{\rho_{L^*}}(\gamma(t_i))| \leq \eps |\gamma(t')-\gamma(t)|=\eps|t_{i+1}-t_i|.
	\end{equation}
Using this we obtain,
 \begin{equation}\begin{split}
	|\frac {\rho_{K^*}}{\rho_{L^*}}(u_1)-\frac {\rho_{K^*}}{\rho_{L^*}}(u_0)|= |\sum_{i=1}^{k-1}\frac {\rho_{K^*}}{\rho_{L^*}}\big( \gamma({t_{i+1}})\big)-\frac {\rho_{K^*}}{\rho_{L^*}}\big( \gamma(t_i)\big)| \leq \\ \sum_{i=1}^{k-1} \eps |t_{i+1}-t_i|= \\ \eps 
\end{split}
	\end{equation}
	Since $\eps $ was arbitrary, we obtain that $\frac {\rho_{K^*}}{\rho_{L^*}}$ is constant on $D$, which shows the desired.
	\end{proof}

 \smallskip

\frenchspacing
\bibliographystyle{cpam}

\end{document}